\documentclass[11pt,reqno]{amsart}%svjour3

\usepackage{amsmath}
\usepackage{amssymb}
\usepackage{graphicx}
\input xy

\newtheorem{thm}{Theorem}[section]
\newtheorem{prop}[thm]{Proposition}
\newtheorem{lem}[thm]{Lemma}
\newtheorem{cor}[thm]{Corollary}
\newtheorem{defn}[thm]{Definition}
\newtheorem{rema}{Remark}[section]

\def\mb{\mathbf}

\def\mr{\mathrm}
\def\mc{\mathcal}

\newlength{\equwidth}
\newcommand{\crd}{\mbox{$                                     
\begin{picture}(9,8)(1.6,0.15)
\put(1,0.2){\mbox{$ D \hspace{-7.8pt} /$}}
\end{picture}$}}
\def\dirac{\crd}

\DeclareMathOperator{\PP}{P}

\def\A{\mathcal A}
\def\D{\mathcal D}

\def\X{\mathfrak X}
\def\al{\alpha}

\def\ga{\gamma}
\def\de{\delta}

\def\rh{\rho}

\def\si{\sigma}

\def\Ups{\Upsilon}
\def\ph{\varphi}

\def\om{\omega}
\def\Ga{\Gamma}
\def\De{\Delta}
\def\Th{\Theta}
\def\La{\Lambda}

\def\Om{\Omega}

\def\t{\otimes}

\def\goesto{\rightarrow}

\def\embed{\hookrightarrow}

\def\squig{\rightsquigarrow}

\def\cinf{\ensuremath{\mathrm{C}^\infty}}

\def\na{\mathrm{\nabla}}

\def\G{\mathcal{G}}

\def\D{\mathcal{D}}

\DeclareMathOperator{\End}{End}
\DeclareMathOperator{\GL}{GL}
\DeclareMathOperator{\Spin}{Spin}
\DeclareMathOperator{\CSpin}{CSpin}
\DeclareMathOperator{\Sp}{Sp}
\def\CO{\mathrm{CO}}

\def\ti{\tilde}

\def\rr{\ensuremath{\mathbb{R}}}
\def\cc{\ensuremath{\mathbb{C}}}

\def\SL{\ensuremath{\mathrm{SL}}}

\def\SO{\ensuremath{\mathrm{SO}}}

\def\GL{\ensuremath{\mathrm{GL}}}

\def\so{\ensuremath{\mathfrak{so}}}

\def\g{\ensuremath{\mathfrak{g}}}

\def\p{\ensuremath{\mathfrak{p}}}

\def\ce{\ensuremath{\mathcal{E}}}
\newcommand{\bg}{\mbox{\boldmath{$ g$}}}

\DeclareMathOperator{\Hol}{Hol}

\def\today{\ifcase\month\or
 January\or February\or March\or April\or May\or June\or
 July\or August\or September\or October\or November\or December\fi
 \space\number\day, \number\year}

 \usepackage{amssymb}
\usepackage{paralist} \usepackage{graphicx} \usepackage{hyperref}
\usepackage[all]{xy}

\def\mb{\mathbf} 

\def\mr{\mathrm}
\def\mc{\mathcal}

\title{The twistor spinors of generic $2$- and $3$-distributions} \author{Matthias
Hammerl and Katja Sagerschnig} \email{matthias.hammerl@univie.ac.at,
katja.sagerschnig@univie.ac.at} \address{Faculty of Mathematics,
University of Vienna, Nordbergstra\ss e~15, A--1090 Wien, Austria}
\date{\today}


\begin{document}

\maketitle

\begin{abstract} 
Generic distributions on $5$- and $6$-manifolds give
rise to conformal structures that were discovered by P. Nurowski
resp. R. Bryant.
We describe both as Fefferman-type constructions and
show that for orientable distributions one obtains conformal spin
structures.
The resulting conformal spin geometries are then characterized
by their conformal holonomy and equivalently by the existence of a
twistor spinor which satisfies a genericity condition.
Moreover, we show that given such a  twistor spinor
we can decompose a conformal Killing field of the structure.
We obtain explicit formulas relating conformal Killing fields,
 almost Einstein structures  and  twistor spinors.
% \subclass{34A26, 35N10, 53A30, 53B15, 53B30} 
% \keywords{generic distributions, conformal geometry, spin geometry, twistor spinors,
% Fefferman-type constructions, conformal Killing fields, almost Einstein
% scales}
\end{abstract}

\section{Introduction}

It was found by P. Nurowski \cite{nurowski-metric} that one can naturally associate to
every generic $2$-distribution on a $5$-manifold a conformal structure
of signature $(2,3)$, and a similar observation has been made by
R. Bryant \cite{bryant-3planes}, who showed that there is a natural
conformal $(3,3)$-structure associated to every generic $3$-distribution in
dimension $6$.
These treatments employ  Cartan's method of equivalence and explicitly
prolong the equations defining the distributions to a \emph{Cartan\ connection\
form}, which is then seen to induce a \emph{conformal\ class\ of\ metrics}
on the underlying manifold. \cite{nurowski-metric}\ and \cite{bryant-3planes}\
also show that the induced conformal structures have special
holonomies $G_2\subset\SO(3,4)$ and $\SO(3,4)\subset \SO(4,4)$. The necessary computations are not
easily accessible and quite complicated when done in full detail.
The approach taken here is to deal with these two constructions by descriptions  as \emph{Fefferman-type} constructions \cite{cap-constructions}. This viewpoint is useful for focusing on the essential algebraic relations
between the structure groups of the geometries in question.
The treatment via parabolic geometry conveniently shows that conformal structures associated to these distributions admit
non-trivial solutions to certain overdetermined systems of PDEs and it uncovers relations between solution spaces - this is
 similar to the classical Fefferman spaces \cite{cap-gover-holonomy-cr,cap-gover-cr-tractors}.

The purpose of this text is twofold. First, it provides a complete discussion of the Fefferman-type construction for a generic rank
 $3$-distribution in dimension $6$. Second, it details and extends a relation to spin geometry that was  found in \cite{mrh-thesis} for  generic
rank $2$-distributions: We show that the twistor spinor attached to a generic rank $2$ or $3$ distribution is a very convenient encoding 
of the  distribution and makes the special properties of the induced conformal structure easily visible. 
The Fefferman-type construction for
generic rank $2$ distributions has been treated by the authors in \cite{mrh-sag-rank2}. The new treatment here via the relation to spin geometry
yields considerable simplifications.

The structure of this paper is as follows:
In section 2 we discuss conformal spin structures of signature $(2,3)$ and $(3,3)$\ and
how twistor spinors satisfying a genericity condition on such structures give rise
to generic distributions and conformal holonomy reductions.\\
In section 3 we show that \emph{all} conformal spin structures of signature $(2,3)$\ and $(3,3)$
admitting such a generic twistor spinor are induced by a Fefferman-type construction
from a generic rank $2$- resp. $3$- distribution.\\
In section 4 we show that this twistor spinor can be used to decompose the conformal
Killing fields of the induced conformal spin structure into a symmetry of the distribution
and, in one case an almost Einstein scale, and in the other case a twistor spinor which is orthogonal in a suitable sense.

\subsection*{Acknowledgements}
As always, the authors have benefited from many discussions with Andreas \v{C}ap. Both authors gladly acknowledge support from
 project P 19500-N13 of the "Fonds zur
F\"{o}rderung der wissenschaftlichen Forschung" (FWF).
In the addition, the first author was supported by the IK I008-N funded by the
University of Vienna, and the second author was supported by
a L'Or\'{e}al Fellowship "For Women in Science".

\section{Generic twistor spinors on conformal spin structures of signature $(2,3)$\ and $(3,3)$}

\subsection{Conformal spin structures}\label{conformalspinstr}

A\emph{ conformal\ structure} of
signature $(p,q)$ on an $n=p+q$-dimensional
manifold $M$ is an equivalence class $\mc{C}$ of
pseudo-Riemannian metrics with two metrics $g$ and $\hat{g}$ being
equivalent if $\hat{g}=\mr{e}^{2f}g$ for a function $f\in
C^{\infty}(M)$.  Suppose we have a manifold with a conformal structure
of signature $(p,q)$.  Let $\mathcal{G}_0$ be the associated conformal
frame bundle with structure group the conformal group $\CO_o(p,q)=\rr_+\times \SO_o(p,q)$
preserving both orientations.  Then a \emph{conformal spin structure}
on $M$ is a reduction of structure group of $\mathcal{G}_0$ to
$\CSpin(p,q)=\mathbb{R}_{+} \times \Spin(p,q)$.

It will be useful to employ abstract index
notation, \cite{penrose-rindler-87}: we write $\ce_a=\Om^1(M),\ce^a=\X(M)$ and
multiple indices as in $\ce_{ab}=T^*M\t T^*M$\ denote tensor products.
To write conformally invariant objects we will also need the \emph{conformal density bundles} $\ce[w]$,
which are the line
bundles associated to the $1$-dimensional representations $(c,C)\mapsto
c^r\in\rr_+$\ of $\CSpin(p,q)=\rr_+\times \Spin(p,q)$.  The tensor
product of a bundle $\mc{V}$\ with $\ce[w]$\ will be denoted
$\mc{V}[w]$. 
We note that the conformal class of metrics $\mc{C}$\ gives rise to a
canonical \emph{conformal metric} $\bg\in\ce_{(ab)}[2]$\ which is used
to identify $\ce^a$\ with $\ce_a[2]$.

Let us briefly introduce the main curvature quantities of the conformal
structure $\mc{C}$ (cf. \cite{eastwood-notes-conformal}.)
For $g\in\mc{C}$, let
\begin{align*}
  \mr{P}_g:=\frac{1}{n-2}(\mr{Ric}_g-\frac{\mr{Sc}_g}{2(n-1)}g)
\end{align*}
be the \emph{Schouten\ tensor}; this is a trace modification
of the Ricci curvature $\mr{Ric}_g$\ by a multiple
of the scalar curvature $\mr{Sc}_g$. The trace of
the Schouten tensor is denoted $J_g=g^{pq}\PP_{pq}$.
We will omit the
subscripts $g$\ hereafter when giving a formula
with respect to some $g\in\mc{C}$.

It is well known that the complete obstruction against
conformal flatness of $(M,\mc{C})$\ with, $\mc{C}$\ having
signature $p+q\geq 3$, is the \emph{Weyl curvature}
\begin{align*}
  C_{ab\; d}^{\;\;\; c}:=R_{ab\; d}^{\;\;\; c}-2\de_{[a}^c\PP_{b]d}+2\bg_{d[a}\PP_{b]}^{\; c},
\end{align*}
where indices between square brackets are skewed over.

We now introduce the basic ingredients of \emph{tractor calculus}\ for
conformal structures (\cite{thomass}). The tractor bundles are to be introduced
as equivalence classes now, and alternatively as associated bundles to
the Cartan structure bundle of a conformal spin structure in the next section \ref{Feffermanconstruction}.

\subsubsection{The standard tractor bundle}\label{secstd}
The standard tractor bundle $\mc{T}$ of a conformal structure $(M,\mc{C})$\ is
defined as an equivalence class of bundles $[\mc{T}]_g,g\in \mc{C}$:
For a given metric $g$\ in the conformal class, we define the
direct sum bundle
$
  [\mc{T}]_g:=
  \begin{pmatrix}
    \ce[-1]\\
    \ce_a[1] \\
    \ce[1]
  \end{pmatrix},
$
and a section
$
  [s]_g=
  \begin{pmatrix}
    \rh \\
    \ph_a \\
    \si
  \end{pmatrix}
  \in \Ga([\mc{T}]_g)
$
corresponds
to the section
$
          [s]_{\hat g}=
        \begin{pmatrix}
            \hat\rh \\
            \hat\ph_a \\
            \hat\si
        \end{pmatrix}=
        \begin{pmatrix}
            \rh-\Ups_a\ph^a-\frac{1}{2}\si\Ups^b\Ups_b \\
            \ph_a+\si \Ups_a  \\
            \si
        \end{pmatrix}
$
for $\hat g=\mr{e}^{2f}g$, $\Ups=df$.
The standard tractor bundle carries the invariant tractor metric
$
  [\bf{h}]_g=
  \begin{pmatrix}
    0 & 0 & 1 \\
    0 & \bg & 0 \\
    1 & 0 & 0
  \end{pmatrix},
$
which is compatible with the standard tractor connection
\begin{align}\label{sttracon}
  [\na^{\mc{S}}_c 
  \begin{pmatrix}
    \rh \\
    \ph_a \\
    \si
  \end{pmatrix}]_g
  =
  \begin{pmatrix}
    D_c \rh-\mr{P}_{c}^{\; b}\ph_b \\
    D_c\ph_a+\si \mr{P}_{ca}+\rh \bg_{c a}\\
    D_c \si-\ph_c
  \end{pmatrix}.
\end{align}

\subsubsection{The spin tractor bundle}
In the case where $(M,\mc{C})$\ is a conformal spin structure, we
define the \emph{weighted\ conformal spin bundle}\ of $(M,\mc{C})$\
as
\begin{align*}
  S[\frac{1}{2}]:=\G_0\times_{\CSpin(p,q)}\De^{p,q}[\frac{1}{2}].
\end{align*}
Then we have the \emph{conformal Clifford symbol}
$
  \ga\in \Ga(\End(TM) \t S[1]).
$
For $\xi\in\X(M)$\ and $\chi\in\Ga(S[\frac{1}{2}])$\ we will write Clifford multiplication also as
$\xi\cdot\chi=\ga(\xi)\chi$.

We define the \emph{spin tractor bundle}\ of $\mc{C}$\ again
as an equivalence class over $\mc{C}$:
with respect to $g$\ it is the direct sum of weighted spin bundles
$
  [\mc{S}]_g:=
  \begin{pmatrix}
    S[-\frac{1}{2}] \\
    S[\frac{1}{2}]
  \end{pmatrix},
$
and a section
$
  [X]_g=
  \begin{pmatrix}
    \tau \\
    \si
  \end{pmatrix}
  \in [\mc{S}]_g
$
corresponds to
$
  [X]_{\hat{g}}=
  \begin{pmatrix}
    \tau+\frac{1}{2}\PP_{cp}\ga^p \chi\\
    \chi
  \end{pmatrix}
  \in [\mc{S}]_{\hat{g}}.
$
for $\hat g=\mr{e}^{2f}g$, $\Ups=df$.

Indeed, $\mc{S}$\ is the Clifford representation of $\mc{T}$:
This is seen by introducing the Clifford action
\begin{align}\label{traCli}
    \begin{pmatrix} \rh \\ \ph_a \\ \si
  \end{pmatrix} \cdot
  \begin{pmatrix} \tau \\ \chi
    \end{pmatrix}
    =
  \begin{pmatrix} -\ph_a\cdot\tau-\sqrt{2}\rh \chi \\
\ph_a\cdot\chi+\sqrt{2}\si \tau
  \end{pmatrix}.
\end{align}
It is easy to compute directly (see e.g.\cite{mrh-thesis}) that indeed
$s\cdot s\cdot X=-\mb{h}(s,s)X$\ for all $s\in\Ga(\mc{T}),X\in\Ga(\mc{S})$,
and that this action is well defined.

$\mc{S}$\ carries the spin tractor connection that is induced from
the standard tractor connection on $\mc{T}$:
\begin{align*}
  [\na^{\mc{S}}_c
  \begin{pmatrix}
    \tau \\
    \chi
  \end{pmatrix}
  ]_g
  =
  \begin{pmatrix}
    D_c\tau+\frac{1}{\sqrt{2}}\PP_{cp}\ga^p\chi\\
    D_c\chi+\frac{1}{\sqrt{2}}\ga_c\chi
  \end{pmatrix}.
\end{align*}

\begin{defn} \label{confhol}
The \emph{conformal holonomy} of a conformal spin structure $\mc{C}$ is defined as 
  \begin{align}\label{defhol}
\Hol(\mc{C}):=\Hol(\nabla^{\mc{T}})=\Hol(\nabla^{\mc{S}})\subset\Spin(p+1,q+1).
\end{align}
\end{defn}

\subsection{The twistor spinor equation and its prolonged form}

For a given metric $g\in\mc{C}$\ the \emph{Dirac operator} is
defined as
\begin{align*}
  &\crd:\Ga(S)\goesto \Ga(S),\ 
  \crd:=\ga^pD_p,
\end{align*}
and is used to define the \emph{twistor operator} (cf. e.g. \cite{baum-friedrich-twistors})
\begin{align*}
  &\Th:\Ga(S[\frac{1}{2}])\goesto \Ga(T^*M\t S[\frac{1}{2}]),\\
  &\Th(\chi):=D\chi+\frac{1}{n}\gamma\crd\chi;
\end{align*}
Alternatively, $\Th$\ is described as the projection of the
Levi-Civita derivative of a spinor to the kernel of Clifford multiplication.
The definition of $\Th$\ with respect to the
appropriately weighted spinor spaces used here is independent of the
choice of $g\in\mc{C}$, i.e., it is a \emph{conformally invariant} linear
differential operator.
An element in the kernel of $\Th$\ is called a \emph{twistor spinor},
and we denote the space of twistor spinors by $\mb{Tw}_{\mc{C}}:=\ker\Th$.
By the following result, the space of twistor spinors is always
finite dimensional:

\begin{prop}[\cite{friedrich-conformalrelation},\cite{baum-friedrich-twistors},\cite{branson-spin},\cite{felipe-habil},\cite{mrh-thesis}]
\label{prop-spin} 

Let $\Pi_0:\mc{S}\goesto\De[\frac{1}{2}]$\ be the (well-defined) projection from
the spin tractor bundle to the lowest slot.

$\Pi_0$\ induces an isomorphism of the space of $\na^{\mc{S}}$-parallel
sections of $\mc{S}$\ with the space of twistor spinors $\mb{Tw}_{\mc{C}}$\ in
$\Ga(S[\frac{1}{2}])$.
Its inverse is the conformally invariant differential splitting operator
$L_0^{\mc{S}}:\Ga(S[\frac{1}{2}])\goesto \Ga(\mc{S})$\ that
is defined with respect to a $g\in\mc{C}$\ by
\begin{align}\label{twisplit} 
   \chi\mapsto
  \begin{pmatrix} 
    \frac{\sqrt{2}}{n}\crd\chi \\ 
    \chi
  \end{pmatrix}.
\end{align}
\end{prop}

In particular, one sees that the existence of a twistor spinor
reduces the holonomy of the conformal spin structure.
\
\subsection{Conformal spin structures of signature $(2,3)$}\label{sec23}
We take some $g_{2,3}$\ be some signature $(2,3)$-bilinear form on $\rr^5$\
and write $\rr^{2,3}$\ for $\rr^5$\ endowed with this form.
It is well known (due to \cite{cartan-spinors}, also cf. \cite{lounesto-clifford},\cite{korman-sparling-fierz}) that
the real, $4$-dimensional spin representation
$\De^{2,3}$\ of $\Spin(2,3)=\Spin(g_{2,3})$
carries a unique skew-form $b_{2,3}\in\La^2(\De^{2,3})^*$
that satisfies
\begin{align}\label{compatb23}
  b_{2,3}(\xi\cdot \chi,\tau)-b_{2,3}(\chi,\xi\cdot \tau)=0
\end{align}
for all $\xi\in\rr^{2,3},\chi,\tau\in\De^{2,3}$.
It follows in particular that $b_{2,3}$\ is invariant under $\Spin(2,3)$,
and this realizes the isomorphism $\Spin(2,3)\cong\Sp(4,\rr)$.
The corresponding skew-symmetric pairing to $b_{2,3}$\ on the
$\CSpin(2,3)$-associated spin bundles is denoted
$\mb{b}_{2,3}\in\Ga(\La^2 S^*[-1]).$

\begin{defn}
  A twistor spinor $\chi\in\Ga(S[\frac{1}{2}])$ is called \emph{generic}\
  if it satisfies
     $ \mb{b}_{2,3}(\chi,\dirac\chi)\not=0.$
  \end{defn}

\begin{rema}
  \begin{enumerate}\itemsep=0pt
  \item 
It easily follows from the twistor spinor equation, that
$\mb{b}_{2,3}(\chi,\dirac\chi)\in\cinf(M)$\ is constant, given
that $M$\ is connected, which we will assume. It is also easy to see from skew-symmetry
of $\mb{b}_{2,3}$\ and the transformation behavior
of spin tractors that this number doesn't depend on the choice of
$g\in\mc{C}$\ used for its computation.
\item 
 Recall that a spinor is called \emph{pure} if its kernel under
 Clifford multiplication is a maximally isotropic subspace - in our case this means
 that it has dimension $2$. Since $\Spin(2,3)$\ acts transitively
 on $\De^{2,3}\backslash\{0\}$, all non-zero spinors are pure.
 It is a well known fact due to E. Cartan \cite{cartan-spinors}\ that
 the canonical pairing between two pure spinors is non-trivial if and only if
 their kernels under Clifford multiplication are transversal. This
 is easily seen directly in this case, which is done in the proof below.
  \end{enumerate}
\end{rema}

\begin{prop}\label{propframe23}
  Let $\chi\in\Ga(S[\frac{1}{2}])$\ be a generic twistor spinor on
  a conformal spin structure $(M,\mc{C})$\ of signature $(2,3)$.
  Denote $\tau=\frac{\sqrt{2}}{5}\dirac\chi$\ for some $g\in\mc{C}$,
  and assume that $\mb{b}_{2,3}(\chi,\tau)=1$.
  \begin{enumerate}\itemsep=0pt
  \item 
  For every $x\in M$\ there is a local frame
  $e_1,e_2,r,f_1,f_2\in\X(U)$, $U$\ a neighborhood of $x$,
  such that on $U$,
  \begin{align}\label{frame1}
    &\ker\ga\chi=\mr{span}(e_1,e_2), \ker\ga\tau=\mr{span}(f_1,f_2),\\\notag
    &(\ker\ga\chi)^{\bot}\cap(\ker\ga\tau)^{\bot}=\rr r,
  \end{align}
  \begin{align}\label{frame2}
    &g(r,r)=-1,\ g(e_i,r)=0,\ g(f_i,r)=0,\ g(e_i,f_j)=\de_{ij},
  \end{align}
  and
\begin{align}\label{frame3}
  &\frac{1}{2}e_1\cdot e_2\cdot\tau=\chi,\ \frac{1}{2}f_1\cdot f_2\cdot\chi=\tau.
\end{align}
  This implies that
\begin{align}\label{frame4}
  &r\cdot\chi=\chi,\ r\cdot\tau=\tau,\ 
  f_1\cdot\chi=-e_2\cdot\tau,\ f_2\cdot\chi=e_1\cdot\tau.
\end{align}
  \item
For $\xi\in\X(M)$\ arbitrary and $\eta\in\ker\ga\chi\subset\X(M)$,
one has
\begin{align}
  &\mb{b}_{2,3}(\xi\cdot\chi,\tau)=-g(\xi,r),\label{formulab23r} \\
  &\mb{b}_{2,3}(\xi\cdot\chi,\eta\cdot\tau)=-2g(\xi,\eta). \label{formulab23}
\end{align}
  \end{enumerate}
  \begin{proof}
    \begin{enumerate}\itemsep=0pt
    \item
      Let $y\in M$\ be arbitrary. We claim that $\ker\ga\chi(y)$\
      and $\ker\ga\tau(y)$\ have transversal ($2$-dimensional) kernels under 
      Clifford multiplication: assume that there exists some $0\not=\xi\in\ker\ga\chi(y)\cap\ker\ga\tau(y)$. Then we can take some isotropic $\eta$\ with $g(\xi,\eta)=1$,
      and then $0=\mb{b}_{2,3}(\xi\cdot\chi,\eta\cdot\tau)=\mb{b}_{2,3}(\eta\cdot\xi\cdot\chi,\tau)=-2\mb{b}_{2,3}(\chi,\tau)$,
      which contradicts the assumption.
 Therefore we can take local sections $e_1,e_2$\ and $f_1,f_2$\ on some
 neighborhood $U$\ of $x\in M$\ which are
 linearly independent at all points in $U$
 and span the kernels of $\chi$\ resp. $\tau$\ under Clifford multiplication.
 It is clear how to choose $r\in\X(U)$\ then and that we can achieve
 \eqref{frame1},\eqref{frame2}. The additional freedom $\al e_1,\frac{1}{\al}f_1$, $\al\in\rr_*$\
 allows us to obtain \eqref{frame3}, and \eqref{frame1}-\eqref{frame3}\ automatically
 imply \eqref{frame4}.
    \item 
      To see \eqref{formulab23r}, note that
      \begin{align*}
        &\mb{b}_{2,3}(\xi\cdot\chi,\tau)=\mb{b}_{2,3}(\xi\cdot\chi,r\cdot\tau)=\\
        &=\mb{b}_{2,3}(r\cdot\xi\cdot\chi,\tau)=-2-\mb{b}_{2,3}(\xi\cdot\chi,\tau).
      \end{align*}
      Now to \eqref{formulab23}: Let $\eta\in\ker\ga\chi$.
      If $\xi=e_i,i=1,2$\ the equation holds, since $\xi\cdot\chi=0$\
      and $g(\xi,\eta)=0$. If $\xi=r$, then $\xi\cdot\chi=\chi$,
      and by \eqref{compatb23} and \eqref{frame2} the equation holds.
      If $\xi=f_i,i=1,2$, then 
      \begin{align*}
        &\mb{b}_{2,3}(\xi\cdot \chi,\eta\cdot \tau)=\mb{b}_{2,3}((\xi+\eta)\cdot \chi,(\xi+\eta)\cdot\tau)=\\&=\mb{b}_{2,3}((\xi+\eta)\cdot(\xi+\eta)\cdot\chi,\tau)=-2g(\xi,\eta).
      \end{align*}
    \end{enumerate}
  \end{proof}
\end{prop}

The genericity of a twistor spinor $\chi\in\Ga(S[\frac{1}{2}])$\
carries over to the induced distribution $\mc{D}_{\chi}:=\ker\ga\chi.$
To make this precise, we first define,
for two subbundles $\D_1\subset TM$ and $\D_2\subset TM$,
\begin{gather}\label{distspan}
  [\D_1,\D_2]_x:=\mr{span}(\{[\xi,\eta]_x :\ \xi\in\Ga(\D_1),\eta\in\Ga(\D_2)\}).
\end{gather}

\begin{defn}\label{defgen2}
  A smooth rank $2$\ subbundle $\mc{D}$\ of the tangent bundle $TM$\ of a $5$-manifold $M$\ 
  is called a \emph{generic 2 distribution}\ if $\mc{D}^1:=[\mc{D},\mc{D}]\subset TM$\ is of
  constant rank $3$\ and $[\mc{D},\mc{D}^1]$\ is already $TM$.
\end{defn}

Employing Proposition \ref{propframe23} we can show
\begin{prop}\label{propgen23}
  Let $(M,\mc{C},\chi)$\ be a conformal spin structure of signature $(2,3)$\
  with a generic twistor spinor $\chi$. Then $\mc{D}_{\chi}=\ker\ga\chi$\
  is a generic rank $2$\ distribution on $M$.
  \begin{proof}
    Abbreviate $\tau=\frac{\sqrt{2}}{5}\dirac\chi$. Then
$\crd\chi=-\frac{1}{\sqrt{2}}\tau$.
Choose a local frame $e_1,e_2,r,f_1,f_2$\ in some neighborhood
$U\subset M$\ with the properties \eqref{frame1}-\eqref{frame3}.
Then $e_2\cdot\chi=0$, and
therefore also $\mb{b}_{2,3}(e_2\cdot\chi,\tau)=0$. A differentiation
gives $$\mb{b}_{2,3}(D_{e_1}e_2)\cdot\chi,\tau)-\frac{1}{\sqrt{2}}\mb{b}_{2,3}(e_2\cdot e_1\cdot \tau,\tau)=0$$
and via \eqref{frame3}\ we see 
\begin{align}\label{nae1e2formula}
  \mb{b}_{2,3}(D_{e_1}e_2\cdot\chi,\tau)=-\sqrt{2},\ \mb{b}_{2,3}(D_{e_2}e_1\cdot\chi,\tau)=\sqrt{2}.
\end{align}
It is also easily seen that
\begin{align}\label{nae1e2formulaz}
  \mb{b}_{2,3}(D_{e_1}e_1\cdot\chi,\tau)=\mb{b}_{2,3}(D_{e_2}e_2\cdot\chi,\tau)=0,
\end{align}
which will be needed below.
Alternating $e_1,e_2$\ in \eqref{nae1e2formula} shows $\mb{b}_{2,3}([e_1,e_2]\cdot\chi,\tau)=-2\sqrt{2},$
but since
$\mb{b}_{2,3}([e_1,e_2]\cdot\chi,\tau)=-g([e_1,e_2],r),$ we
have $g([e_1,e_2],r)=2{\sqrt{2}}.$
A similar calculation gives that for all $\eta\in\Ga(\ker\ga\chi)$\ one
has $g([e_1,e_2],\eta)=0$, and thus $[e_1,e_2]\in\Ga(\ker\ga\chi)$,
and then $[e_1,e_2]/\mc{D}=-{\sqrt{2}}r/\mc{D}.$

Starting with the equation $\mb{b}_{2,3}(r\cdot\chi,\eta\cdot\tau)=0$\ for
$\eta\in\Ga(\ker\ga\chi)$\ and differentiating in direction
$\xi\in\Ga(\ker\ga\chi)$
one obtains 
\begin{align*}
  \mb{b}_{2,3}(D_{\xi}r\cdot\chi,\eta\cdot\tau)-\frac{1}{\sqrt{2}}\mb{b}_{2,3}(r\cdot\xi\cdot\tau,\eta\cdot\tau)+\mb{b}_{2,3}(\chi,D_{\xi}\eta\cdot\chi)=0.
\end{align*}
Reversing the roles of $r$\ and $\xi$\ gives
\begin{align*}
  \mb{b}_{2,3}(D_r\xi\cdot\chi,\eta\cdot\tau)-\frac{1}{\sqrt{2}}\mb{b}_{2,3}(\xi\cdot r\cdot\tau,\eta\cdot\tau)=0.
\end{align*}
One has, since $\mb{b}_{2,3}$\ is skew and satisfies \eqref{compatb23},
$\mb{b}_{2,3}(r\cdot\xi\cdot\tau,\eta\cdot\tau)=\mb{b}_{2,3}(\xi\cdot r\cdot\tau,\eta\cdot\tau)=0$.
Therefore a subtraction yields
\begin{align*}
  \mb{b}_{2,3}([\xi,r]\cdot\chi,\eta\cdot\tau)=-\mb{b}_{2,3}(D_{\xi}\eta\cdot\chi,\tau).
\end{align*}
Employing \eqref{formulab23}, \eqref{nae1e2formula} and \eqref{nae1e2formulaz}\ then gives
\begin{align*}
  g([e_1,r],e_1)=0,\ g([e_1,r],e_2)=\frac{1}{\sqrt{2}},
\end{align*}
and thus 
$  [e_1,r]/\mc{D}^1=\frac{1}{\sqrt{2}} f_2/\mc{D}^1,$
and similarly
$
  [e_2,r]/\mc{D}^1=-\frac{1}{\sqrt{2}} f_1/\mc{D}^1.$
  \end{proof}
\end{prop}

We now calculate the holonomy reduction
implied by the existence of a generic twistor spinor
on a conformal spin structure of signature $(2,3)$.

\begin{thm}\label{thm-holchar23} Let $(M,\mc{C})$\ be a conformal spin
structure of signature $(2,3)$\ with its real 4 dimensional,
conformally weighted spin bundle $S[\frac{1}{2}]$, which is endowed
with its skew-form $\mb{b}_{2,3}$.

Then $\Hol(\mc{C})\subset G_2\subset\Spin(3,4)$\ if and only if there exists a twistor spinor
$\chi\in \Ga(S[\frac{1}{2}])$\ with $\mb{b}_{2,3}(\chi,\dirac\chi)\not=0$.
\end{thm}

Here, and throughout this text, $G_2$\ will always
refer to the connected Lie group with fundamental group $\mathbb{Z}_2$\ that
has Lie algebra $\g_2$, the split real form of the exceptional
complex Lie group $\g_2^{\cc}$. 

By Proposition \ref{prop-spin}\ and the definition of conformal
holonomy \ref{defhol}, to see Theorem \ref{thm-holchar23}, we only need that the stabilizer in
$\Spin(3,4)$\ of an element
$
X=
\begin{pmatrix}
  \tau \\
  \chi
\end{pmatrix}
\in\De^{3,4}$
with the property that $b_{2,3}(\chi,\tau)\not=0$\ is indeed $G_2$.
This is the content of the next subsection.

\subsubsection{Algebraic background on $G_2\embed\Spin(3,4)$}\label{algembed}
Recall that we have fixed some signature $(2,3)$-bilinear form $g_{2,3}$ on $\rr^5$,
and we write $\rr^{2,3}=(\rr^5,g_{2,3})$.
Let us extend this form orthogonally to a form $h_{3,4}$\ on $\rr^7$\
by introducing two new directions $e_+$\ and $e_-$\
and defining $h_{3,4}=2(de_+)(de_-)+g_{2,3}$.
The vectors $e_+,e_-\in\rr^{3,4}=(\rr^7,h_{3,4})$\ are isotropic,
and $\rr^{2,3}\embed\rr^{3,4}$\ is the orthogonal complement to the subspace
of $\rr^{3,4}$\ spanned by $e_+$\ and $e_-$.

We have
$
  \De^{3,4}=
  \begin{pmatrix}
    \De^{2,3} \\
    \De^{2,3}
  \end{pmatrix},
$
where an element $v=\si e_-\oplus \xi\oplus\rh e_+ \in \rr e_-\oplus\rr^{2,3}\oplus\rr e_+=\rr^{3,4}$
acts on 
 $
   X=
\begin{pmatrix} 
  \tau   \\ 
  \chi
\end{pmatrix}
$
by
$  v\cdot X =  
  \begin{pmatrix} 
    -\xi\cdot\tau-\sqrt{2}\rh \chi \\
    \xi\cdot\chi+\sqrt{2}\si \tau
  \end{pmatrix}.
$
$\De^{3,4}$\ is endowed with the canonical signature $(4,4)$ symmetric bilinear form
$B_{3,4}$, defined by
\begin{align}\label{formb34}
  B_{3,4}(
  \begin{pmatrix}
    \tau \\
    \chi
  \end{pmatrix}
,
  \begin{pmatrix}
    \tau' \\
    \chi'
  \end{pmatrix}
)=b_{2,3}(\chi,\tau')+b_{2,3}(\chi',\tau),
\end{align}
and with respect to $B_{3,4}$\ an element  $
   X=
\begin{pmatrix} 
  \tau   \\ 
  \chi
\end{pmatrix}
$\ is evidently non-null if and only if $b_{2,3}(\chi,\tau)\not=0$.

\begin{prop}\label{propg2}
Let $X=
\begin{pmatrix}
  \tau \\
  \chi
\end{pmatrix}
\in\De^{3,4}$\
be such that $b_{2,3}(\chi,\tau)=1$.
\begin{enumerate}\itemsep=0pt
\item 
There exists a basis $e_1,e_2,r,f_1,f_2$ of $\rr^{2,3}\subset\rr^{3,4}$
with the properties  \eqref{frame1}-\eqref{frame4}.
\item
The isotropy algebra $\g_2:=\so(3,4)_X$\ is the split real form of $\g_2^{\cc}$. It carries a grading, more precisely, it is the direct sum
$\g_{-3}\oplus\g_{-2}\oplus\g_0\oplus\g_1\oplus\g_2\oplus\g_3,$ where the individual components are spanned by the following elements: $(i=1,2)$
\begin{align*}
  &\g_{-3}=\mr{span}(e_-\wedge f_i),\\
  &\g_{-2}=\mr{span}(e_-\wedge r-\frac{1}{\sqrt{2}}f_1\wedge f_2),\\
  &\g_{-1}=\mr{span}(e_-\wedge e_i+\sqrt{2}r\wedge i_{e_i}f_1\wedge f_2),\\
  &\g_0=\mr{span}(e_1\wedge f_2,e_2\wedge f_1, e_i\wedge f_i+e_+\wedge e_-),\\
  &\g_1=\mr{span}(e_+\wedge f_i-\sqrt{2}r^i_{f_i}e_1\wedge e_2),\\
  &\g_2=\mr{span}(e_+\wedge r+\frac{1}{\sqrt{2}}e_1\wedge e_2),\\
  &\g_3=\mr{span}(e_+\wedge e_i).
\end{align*}
The sum $\mathfrak{p}=\g_0\oplus\g_1\oplus\g_2\oplus\g_3$ forms a parabolic subalgebra. 
\item
$\Spin(3,4)_X=G_2$.
\item
The restriction of the standard representation of $\Spin(3,4)$\
on $\rr^{3,4}$\ to $G_2$\ is the irreducible standard representation
of $G_2$. The restriction of the spin representation of $\Spin(3,4)$\
on $\Delta^{3,4}$\ to $G_2$\ splits into one copy of the standard representation
and $\rr$.
\end{enumerate}
\begin{proof}
  (1) is done as in Proposition \ref{propframe23}, and (2) is
checked directly. \\ To see (3): 
It follows from the Proposition that the isotropy subgroup of $X$ in $\mathrm{Spin}(3,4)$ is a closed subgroup with Lie algebra $\mathfrak{g}_2$. 
One can verify that it is 
the connected subgroup with fundamental group $\mathbb{Z}_2$: The action of $\Spin(3,4)$ on the space of all spinors of a fixed norm $B_{3,4}(X,X)$
is transitive, since the orbit on any non-null spinor must be open by Proposition \ref{propg2} and dimension count. To see that $G_2:=\Spin(3,4)_X$
is connected and has fundamental group $\mathbb{Z}_2$ is then a standard argument using the exact homotopy sequence for the action.
\\
(4)\ is seen easily on the Lie algebra level, and carries over to the connected groups.
\end{proof}
\end{prop}
This is an alternative and more
detailed description of $G_2\subset\Spin(3,4)$\ of \cite{kath}.
\nocite{bryant-exceptional}

\subsection{Conformal spin structures of signature $(3,3)$}\label{sec33}
This section runs closely parallel to the case of signature $(2,3)$\ in section
\ref{sec23}, and is therefore done succintly. 

Let $g_{3,3}$\ be some signature $(3,3)$-form on $\rr^6$, which
is then written $\rr^{3,3}$.
Let $\De^{3,3}$\ be the real, $8$-dimensional spin representation
of $\Spin(3,3)=\Spin(g_{3,3})$. Then $\De^{3,3}=\De^{3,3}_+\oplus\De^{3,3}_-$
and
the Clifford multiplication is written
\begin{align*}
  \rr^{3,3}\t\De^{3,3}_{\pm}\goesto\De^{3,3}_{\mp},\  \xi\t\chi\mapsto \xi\cdot\chi.
\end{align*}
The canonical pairing $b_{3,3}$\ (cf. e.g. \cite{korman-sparling-fierz}) of $\De^{3,3}_+$\ and $\De^{3,3}_-$
satisfies
\begin{align*}
  b_{3,3}(\chi,\xi\cdot\eta)+b_{3,3}(\eta,\xi\cdot\chi)=0
\end{align*}
for all $\xi\in\rr^{3,3}, \chi,\eta\in\De_{\pm}^{3,3}$.
It follows in particular that $b_{3,3}$\ is invariant under $\Spin(3,3)$,
and this realizes the isomorphism $\Spin(3,3)\cong\SL(4,\rr)$.

The corresponding skew-symmetric pairing to $\mb{b}_{3,3}$\ on the
$\CSpin(3,3)$-associated spin bundles is denoted
$\mb{b}_{3,3}\in \Ga(S_+^*\t S_-^*[-1]).$

\begin{defn}
  A positive twistor spinor $\chi\in\Ga(S_+[\frac{1}{2})$ is called \emph{generic}\
  if it satisfies
  $
    \mb{b}_{3,3}(\chi,\dirac\chi)\not=0.
  $
\end{defn}

Similarly to Proposition \ref{propframe23}\ one shows
\begin{prop}\label{propframe33}
  Let $\chi\in\Ga(S_+[\frac{1}{2}])$\ be a positive generic twistor spinor on
  a conformal spin structure $(M,\mc{C})$\ of signature $(3,3)$.
  Denote $\tau=\frac{\sqrt{2}}{6}\dirac\chi$\ for some $g\in\mc{C}$,
  and assume that $\mb{b}_{3,3}(\chi,\tau)=1$.
  \begin{enumerate}\itemsep=0pt
  \item 
  For every $x\in M$\ there is a local frame
  $e_1,e_2,e_3,f_1,f_2,f_3\in\X(U)$, $U$\ a neighborhood of $x$,
  such that on $U$, 
 \begin{align}\label{framep1}
    \ker\ga\chi=\span(e_1,e_2,e_3),\ \ker\ga\tau=\span(f_1,f_2,f_3),
  \end{align}
  \begin{align}\label{framep2}
    &g(e_i,f_j)=\de_{ij},
  \end{align}
  and
\begin{align}\label{framep3}
  e_1\cdot e_2\cdot e_3\cdot \tau=\chi,\ f_1\cdot f_2\cdot f_3\cdot\chi=\tau.
\end{align}
  \item
For $\xi\in\X(M)$\ arbitrary and $\eta\in\ker\ga\chi\subset\X(M)$,
one has
\begin{align*}
  \mb{b}_{3,3}(\xi\cdot\chi,\eta\cdot\tau)=-2g(\xi,\eta).
\end{align*}
\end{enumerate}
\end{prop}

\begin{defn}\label{defgen3}
  A generic rank $3$\ distribution on a $6$-manifold $M$\ is
  a smooth rank $3$ subbundle $\mc{D}$\ of $TM$\ with
  $[\mc{D},\mc{D}]=TM$.
\end{defn}

An analogous proof to Proposition \ref{propgen23} gives
\begin{prop}
  Let $(M,\mc{C},\chi)$\ be a conformal spin structure of signature $(3,3)$\
  with a positive generic twistor spinor $\chi$. Then $\mc{D}_{\chi}=\ker\ga\chi$\
  is a generic rank $3$\ distribution on $M$.
\end{prop}

We now discuss the conformal holonomy discussion induced
by a generic twistor spinor in signature $(3,3)$:
\begin{thm}\label{thm-holchar33} Let $(M,[g])$\ be a conformal spin
structure of signature $(3,3)$\ with its real 8 dimensional
conformally weighted spin bundle
$S[\frac{1}{2}]=S_+[\frac{1}{2}]\oplus S_-[\frac{1}{2}]$; there is a
canonical non-degenerate pairing $\mb{b}_{3,3}:S_+[\frac{1}{2}]\t
S_-[\frac{1}{2}]\goesto\rr$.

Then $\Hol([g])\subset \Spin(3,4)$\ if and only if there exists a generic twistor
spinor $\chi\in \Ga(S_+[\frac{1}{2}])$,
i.e., 
${\mb{b}_{3,3}}(\chi,\dirac\chi)\not=0$.
\begin{proof}
  For this, we compute, analogously to the case of signature $(2,3)$\ above,
the stabilizer in $\Spin(4,4)$\ of an element
$
X=
\begin{pmatrix}
  \tau \\
  \chi
\end{pmatrix}
\in\De^{4,4}$
with the property that $b_{3,3}(\chi,\tau)\not=0$.
It is easy to see (cf. formula $\eqref{formb44}$\ below) that
the invariant $\Spin(3,3)$-invariant form $b_{3,3}$ can be used 
to define the canonical $\Spin(4,4)$-invariant form $b_{4,4}$\ on $\De^{4,4}$.
Via triality, this is equivalent to the standard signature
$(4,4)$-inner product $h_{4,4}$\ on $\rr^8$, which is the standard representation
of $\Spin(4,4)$. Since the stabilizer of a non-null element in
the standard representation of $\Spin(4,4)$\ is a standard
embedding of $\Spin(3,4)$ into $\Spin(4,4)$, the stabilizer
of a non-null $X\in\De^{4,4}$\ as above is just such a standard embedding
composed with a triality automorphism of $\Spin(4,4)$.
\end{proof}
\end{thm}

Although not strictly necessary for the computation of the
conformal holonomy in this signature, it will be useful
to discuss the embedding $\Spin(3,4)\embed\Spin(4,4)$\ as
the stabilizer of a non-null $X\in\De^{4,4}$\ also
directly again; this will provide us with an
explicit form of the stabilizing Lie algebra in a canonical basis that
will be useful in the next section.

\subsubsection{$\Spin(3,4)\embed\Spin(4,4)$}

We have the $8$-dimensional Clifford representation $\De^{3,3}$\
of $\rr^{3,3}=(\rr^6,g_{3,3})$\ 
that splits into $\De^{3,3}_+\oplus\De^{3,3}_-$\ under $\Spin(3,3)\cong\SL(4)$.

Let $h_{4,4}$\ be the signature $(4,4)$ symmetric bilinear
form $h_{4,4}=2(de_+)(de_-)+g_{3,3}$\ on $\rr^8=\rr e_+ \oplus\rr^{3,3}\oplus \rr e_-$.
The Clifford representation of $\rr^{4,4}=(\rr^8,h_{4,4})$\ 
is defined on $\De^{4,4}_{\pm}:=\De^{3,3}_{\pm}\oplus\De^{3,3}_{\mp}$\
via
\begin{align*}
    \rr^{4,4}\t\De^{4,4}_{\pm}\goesto \De^{4,4}_{\mp},\
   &\begin{pmatrix} \rh \\ \xi \\ \si
  \end{pmatrix} \cdot
  \begin{pmatrix} \tau \\ \chi
  \end{pmatrix} =
  \begin{pmatrix} -\xi\cdot\tau-\sqrt{2}\rh \chi \\
\xi\cdot\chi+\sqrt{2}\si \tau
  \end{pmatrix}.
\end{align*}
The Clifford-invariant symmetric split signature $(4,4)$-form $B_{4,4}$, is
defined on $\De^{4,4}_+$\ and $\De^{4,4}_-$\ by
\begin{align}\label{formb44}
  B_{4,4}(
  \begin{pmatrix}
    \tau \\
    \chi
  \end{pmatrix},
  \begin{pmatrix}
    \tau' \\
    \chi'
  \end{pmatrix}
  )
  =b_{3,3}(\chi,\tau')+b_{3,3}(\chi',\tau).
\end{align}

Using this one shows
\begin{prop}\label{propso34}
Let 
$X=
\begin{pmatrix}
  \tau\\
  \chi
\end{pmatrix}
\in\De_+^{4,4}$\ be such that $b_{3,3}(\chi,\tau)=1$.
  \begin{enumerate}\itemsep=0pt
  \item There exists a basis $e_1,e_2,e_3,f_1,f_2,f_3$\ of $\rr^{3,3}$ such that \eqref{framep1}-\eqref{framep3}\ hold.
    \item
 The isotropy algebra $\g:=\so(4,4)_X$\ is a realization of $\so(3,4)$. It is
 graded as $\g_{-2}\oplus\g_{-1}\oplus\g_0\oplus\g_1\oplus\g_2,$ where, with $i=1,2,3$,
  \begin{align*}
    &\g_{-2}=\mr{span}(e_-\wedge f_i)\\
    &\g_{-1}=\mr{span}(e_-\wedge e_i-\sqrt{2}i_{e_i}f_1\wedge f_2\wedge f_3)\\
    &\g_{0}=\mr{span}(e_i\wedge f_j,\mr{for}, i\not=j\ \mr{and}\ e_i\wedge f_i+e_+\wedge e_-)\\
    &\g_{1}=\mr{span}(e_+\wedge f_i-\sqrt{2}i_{f_i}e_1\wedge e_2\wedge e_3)\\
    &\g_{2}=\mr{span}(e_+\wedge e_i).
  \end{align*}
The subspace $\mathfrak{p}=\g_0\oplus\g_1\oplus\g_2$ forms a parabolic subalgebra.
\item
$\Spin(4,4)_X=\Spin(3,4)$.
\item
The restriction of the standard representation of $\Spin(4,4)$\ to
$\Spin(3,4)$\ is the spin representation $\Delta^{3,4}$,
as is the restriction of the negative spin representation of $\Spin(4,4)$.
The restriction of the positive spin representation of $\Spin(4,4)$\
to $\Spin(3,4)$\ decomposes into a copy of $\De^{3,4}$\ and $\rr$.
\end{enumerate}
\end{prop}

\section{Conformal spin structures associated to generic $2$\
and $3$ distributions.}\label{Feffermanconstruction}

\subsection{Generic distributions and conformal spin structures as parabolic geometries} 
Let $G$ be a Lie group and $P\subset
G$ a closed subgroup, denote by $\mathfrak{g}$ and $\mathfrak{p}$ the respective 
Lie algebras. A \emph{Cartan geometry} (see e.g. \cite{sharpe}) of type $(G,P)$ is given by a
principal bundle $\mathcal{G}\to M$ with structure group $P$ and a
Cartan connection $\omega\in\Omega^1(\mathcal{G},\mathfrak{g})$, this
is a $P$-equivariant $1$-form that reproduces generators of
fundamental vector fields and defines isomorphisms
$\omega_u:T_u\mathcal{G}\to\mathfrak{g}$ for each $u\in\mathcal{G}.$
The basic example of a Cartan geometry of type $(G,P)$ is the
\emph{homogeneous model}, i.e., the bundle $p:G\to G/P$ equipped with
the Maurer Cartan form. The curvature
$\kappa\in\Omega^2(\mathcal{G},\mathfrak{g})$ of a Cartan geometry,
defined as
$$\kappa(\xi,\eta)=d\omega(\xi,\eta)+[\omega(\xi),\omega(\eta)]\ \mr{for}\ \xi,\eta\in\X(\G)$$
is a complete obstruction to local equivalence with the homogeneous
model. The curvature can equivalently be described as a $P$-equivariant function $\mathcal{G}\to\Lambda^2(\mathfrak{g}/\mathfrak{p})^*\otimes\mathfrak{g}$
and we will often take this point of view.

\emph{Parabolic geometries} are Cartan geometries of type
$(G,P)$ for a semisimple Lie group $G$ and a parabolic subgroup $P\subset G$. There
is by now an big amount of general theory available for geometries
of this type, cf. \cite{cap-slovak-par}. One of the main reasons
for their special importance is that they allow uniform Lie
algebraic  \emph{regularity} and \emph{normality} conditions on the Cartan curvature.
Assuming regularity a parabolic geometry determines a certain underlying
structure, called a regular infinitesimal flag structure.
If the parabolic geometry is also normal it is  uniquely determined by its underlying structure, and 
one obtains an equivalence of categories in this case  (cf. \cite{cap-slovak-par} for
the general statement and earlier versions for particular geometries.).

To describe the underlying structures note that every parabolic subalgebra $\mathfrak{p}$ of a semisimple Lie algebra  $\mathfrak{g}$ determines a grading of the Lie algebra
\begin{align} 
\mathfrak{g}=\mathfrak{g}_{-k}\oplus\cdots\oplus\mathfrak{g}_{-1}\oplus\mathfrak{g}_0\oplus\mathfrak{g}_1\oplus\cdots\oplus\mathfrak{g}_k
\end{align} such that $[\mathfrak{g}_i,\mathfrak{g}_j]\subset\mathfrak{g}_{i+j}$, the negative part $\mathfrak{g}_{-}=\mathfrak{g}_{-1}\oplus\cdots\oplus\mathfrak{g}_k$
is generated by $\mathfrak{g}_{-1}$ and
$\mathfrak{p}=\mathfrak{g}_0\oplus\mathfrak{g}_1\oplus\cdots\oplus\mathfrak{g}_k$. 
The positive part $\mathfrak{p}_{+}=\mathfrak{g}_1\oplus\cdots\oplus\mathfrak{g}_k$ is then a nilpotent ideal in $\mathfrak{p}$, and $\mathfrak{g}_0$
a reductive subalgebra. 
The natural action of the corresponding subgroup  $G_0$  preserves the grading on $\mathfrak{g}$, 
while the parabolic $P$ only preserves  filtration induced by the grading. 

For a parabolic geometry of type $(G,P)$ 
this Lie algebra filtration induces a filtration of the tangent bundle
and, assuming regularity, also some additional structure:
Let $T^{-1}M\subset\cdots\subset T^{-k}M=TM$ be a filtration of the tangent bundle by 
subbundles  that is compatible with forming 
Lie brackets. Then the Lie bracket induces a tensorial bracket called Levi bracket $\mathcal{L}:\mathrm{gr}(TM)\times\mathrm{gr}(TM)\to\mathrm{gr}(TM)$ 
 on the associated graded $\mathrm{gr}(TM)=\bigoplus_i T^{i}M/T^{i+1}M$. Suppose $(\mathrm{gr}(TM),\mathcal{L})$ is a bundle of 
 Lie algebras modelled on the graded Lie algebra $\mathfrak{g}_{-}$ 
and consider  the natural frame bundle $\mathcal{P}$  for  $\mathrm{gr}(TM)$ with structure group the automorphisms
 $\mathrm{Aut}_{gr}(\mathfrak{g}_{-})$ of $\mathfrak{g}_{-}$ that preserve the grading.
 A regular infinitesimal flag structure of type $(G,P)$ consists of such a filtration and a reduction of structure group of the bundle $\mathcal{P}$ with 
respect to $\mathrm{Ad}:G_0\to\mathrm{Aut}_{gr}(\mathfrak{g}_{-}).$

\subsubsection{Generic rank two distributions in dimension
five}\label{secg2} 
Suppose $\mathcal{D}$ is a generic rank $2$-distribution on a $5$-manifold, as defined via  \eqref{distspan}
and Definition \ref{defgen2}.
 Defining $T^{-1}M=\mathcal{D},$ $T^{-2}M=[\mathcal{D},\mathcal{D}]$ and $T^{-3}M=TM$ yields a filtered manifold, such that the Levi bracket defines 
isomorphisms $\Lambda^2 T^{-1}M\to T^{-2}M/T^{-1}M$ and $T^{-1}M\otimes T^{-2}M/T^{-1}M\to T^{-3}M/T^{-2}M.$ In particular, $\mathrm{Aut}_{gr}(\mathfrak{g}_{-})\cong \mathrm{GL}(2,\mathbb{R})$ and the frame bundle for $\mathrm{gr}(TM)$ can be identified with the frame bundle for the distribution.
A reduction of this frame bundle to $\GL_{+}(2,\mathbb{R})$ is the same as an orientation of the distribution.

Now let $G_2$ be the connected Lie group with Lie algebra the split real
form of the simple complex exceptional Lie algebra
$\mathfrak{g}_2^{\mathbb{C}}$ and with fundamental group
$\mathbb{Z}_2$.  
A grading of $\mathfrak{g}_2$  corresponding to a maximal  parabolic subalgebra $\mathfrak{p}$ was introduced in Proposition \ref{propg2}. It is easy to see 
that $(\mathrm{gr}(TM),\mathcal{L})$ from above  is modelled on the negative part $\mathfrak{g}_{-}$ of that grading.
Thus, for a parabolic subgroup $P\subset G_2$ with Lie algebra $\mathfrak{p}$, a regular infinitesimal flag structure of type $(G_2,P)$ 
is the same as a reduction  of the frame bundle of a  generic $2$-distribution $\mathcal{D}$ to the structure group $G_0$.
The usual choice for the parabolic is to define it as the stabilizer $P'$ of the line through the highest weight
vector $v$ in the $7$-dimensional fundamental representation of $G_2$. In that case the reductive subgroup $G'_0\cong \GL(2,\mathbb{R}),$ and the regular infinitesimal flag structure is just a generic rank $2$ distribution.
However, in the context of this paper we will use the \emph{connected} parabolic subgroup $P$, i.e., the stabilizer of the ray $\mathbb{R}_+ v$ through the highest weight vector. Then $G_0$ is isomorphic to the group
$\GL_{+}(2,\mathbb{R})$, and a regular infinitesimal flag structure encodes an \emph{oriented}  distribution. Thus, the equivalence result for parabolic geometries implies:

\begin{prop}
With the above choice of Lie groups, there is an equivalence
of categories between  regular, normal parabolic geometries of type
$(G_2,P)$  and oriented generic rank $2$ distributions on $5$-manifolds
\end{prop}

\subsubsection{Generic rank three distributions in dimension
six}\label{secspin34} Now suppose $\mathcal{D}$ is a  generic rank $3$-distributions on a 
$6$-manifold, i.e., values of sections
$\xi,\eta\in\Gamma(\mathcal{D})$ and their Lie brackets $[\xi,\eta]$
span the tangent bundle $TM$, recall \eqref{distspan} and Definition \ref{defgen3}.
Then the distribution gives rise to the filtration $T^{-1}M=\mathcal{D}\subset T^{-2}M=TM$
such that the Levi bracket $\mathcal{L}:\Lambda^2 T^{-1}M\to TM/T^{-1}M$ is an isomorphism.

In this case there is a grading of $\mathfrak{so}(3,4)$ corresponding to a parabolic $\mathfrak{p}$, described explicitly in Proposition \ref{propso34},
 such that in every point the graded Lie algebra $(\mathrm{gr}(T_xM),\mathcal{L}_x)$ is isomorphic to $\mathfrak{g}_{-}$. 
 Now consider as a group with Lie algebra $\mathfrak{p}$ the  parabolic subgroup $P\subset
\Spin(3,4)$ defined as  the stabilizer of a ray through a highest weight vector in
the real spinor representation $\Delta^{3,4}.$ Then $P$ does not
contain the element $-1$ acting as minus the identity on the spin
representation, and one easily  verifies that  $P$ is connected. It
follows that the $2$-fold covering $\Spin(3,4)\to \SO_0(3,4)$
restricts to a diffeomorphism from $P$ onto the connected component of
the parabolic $P'\subset \SO_0(3,4)$ defined as the stabilizer of an
isotropic $3$-dimensional subspace of $\mathbb{R}^{3,4}.$ In
particular, the Levi subgroup $G_0\subset P\subset \Spin(3,4)$ is seen
to be $\GL_{+}(3,\mathbb{R})$. Thus, invoking the general theory yields:

\begin{prop}
With the above choice of Lie groups, there is an equivalence
of categories between  regular, normal parabolic geometries of type
$(SO(3,4),P)$  and oriented generic rank $3$ distributions on $6$-manifolds
\end{prop}

\subsubsection{Conformal spin structures}\label{secspin}

 A conformal spin structure (see section \ref{conformalspinstr}) can be equivalently described as a normal
parabolic geometry of type $(\Spin(p+1,q+1),\tilde{P}),$ where
$\tilde{P}$ is the stabilizer of a positive ray through a null-vector
in $\mathbb{R}^{p+1,q+1}$.  For this choice of groups the parabolic subgroup
$\tilde{P}$ is connected and the reductive  subgroup
$\tilde{G}_0\subset \tilde{P}$ is precisely $\CSpin(p,q).$

We remark that in this case the Cartan structure bundle $\ti\G\goesto M$\
can be realized as the adapted frame bundle of the standard
tractor bundle $\mc{T}$\ introduced in section \ref{secstd}; i.e.,
it is the frame bundle of $(\mc{T},\mb{h})$\ that additionally
satisfies the canonical filtration of $\mc{T}$, cf.\ \cite{cap-gover-cr-tractors}. Conversely, the standard tractor bundle is the associated bundle
$\mc{T}=\tilde{\mc{G}}\times_{\tilde{P}}\mathbb{R}^{p+1,q+1}.$

\subsubsection{Tractor bundles}\label{sec-tractor}
More generally, there are  associated vector bundles carrying canonical linear connections for all types of parabolic geometries: 

Suppose $(\mathcal{G},\omega)$ is a regular, normal parabolic geometry of type $(G,P)$.
Given a $G$-representation $\mathbb{V}$, one can form the associated bundle $\mathcal{V}=\mathcal{G}\times_{P}\mathbb{V}$. Such a vector
bundle is called a \emph{tractor bundle}. 
Let $\G':=\G\times_PG$\ be the \emph{extended Cartan bundle}, which
is now a $G$-principal bundle over $M$.
Then we can extend $\om$
canonically in an equivariant way to the \emph{extended $G$-principal bundle connection form} $\om'\in\Om^1(\G',\g)$. Since
$\mc{V}=\G\times_P \mathbb{V}$\ can also be written as $\G'\times_G \mathbb{V}$,
we see that $\om'$ induces a linear connection on $\mc{V}$,
which is the \emph{normal tractor connection} $\nabla^{\mathcal{V}}$,
cf. eg. \cite{cap-gover-tractor}.

\subsection{The Fefferman-type constructions $\mc{D}\squig\mc{C}_{\mc{D}}$}\label{sec-fefferman}
We prove that to any oriented generic rank $2$ distribution on a $5$ manifold, and to any oriented rank $3$ distribution on a $6$ manifold
there is an associated conformal spin structure. This is done via a Fefferman-type construction in the sense on A. \v Cap.

\subsubsection{Fefferman-type constructions over the same manifold}

Consider an inclusion of simple Lie groups
$G\hookrightarrow\tilde{G},$ and parabolic subgroups
$\tilde{P}\subset \tilde{G}$, $P=\tilde{P}\cap G$, and suppose the
inclusion induces a diffeomorphism of the corresponding homogeneous
spaces $$G/P\cong\tilde{G}/\tilde{P}.$$ Then there is a is a
functorial construction, see \cite{cap-constructions}, associating to
a parabolic geometry $(\mathcal{G},\omega)$ of type $(G,P)$ a
parabolic geometry of type $(\tilde{G},\tilde{P})$: First one extends
the Cartan bundle to a $\tilde{P}$-principal bundle
$$\tilde{\mathcal{G}}=\mathcal{G}\times_{P}\tilde{P},$$ and then one
shows that there is a unique extension of $\omega$ to a Cartan
connection
$\tilde{\omega}\in\Omega^1(\tilde{\mathcal{G}},\tilde{\mathfrak{g}})$
on $\tilde{\mathcal{G}}$.

It is shown in \cite{doubrov-slovak-inclusions} that there is a very
limited number of Lie group data that give rise to such a
Fefferman-type construction over the same base manifold; indeed, there are only three families of such constructions. We discuss the two constructions that give rise to conformal structures.
In fact, assuming orientability of the distributions, we will see that
we get induced conformal spin structures.

\subsubsection{The Fefferman-type construction $G_2\hookrightarrow \Spin(3,4)$}\label{sec-sig34}
Let $P\subset G_2$ and $\tilde{P}\subset \Spin(3,4)$ be the parabolic subgroups from \ref{secg2} and \ref{secspin}, i.e., the subgroups
 defined as the stabilizer of a positive ray through a null-vector in $\mathbb{R}^{3,4}.$ Then $P=\tilde{P}\cap G_2$.
Since $G_2/P$ and $\Spin(3,4)/\tilde{P}$ are compact, connected
and have the same dimension, the homogeneous spaces are indeed diffeomorphic: we have
$$G_2/P\cong \Spin(3,4)/\tilde{P}\cong S^2\times S^3.$$
The functorial construction discussed above
thus assigns to a parabolic geometry $(\mathcal{G},\omega)$ of type
$(G_2,P)$ a parabolic geometry $(\tilde{\mathcal{G}},\tilde{\omega})$
of type $(\Spin(3,4),P).$ 

\begin{prop} An oriented generic  rank
$2$-distribution on a $5$-manifold $M$ naturally induces a conformal
spin structure of signature $(2,3)$ on $M$.
\end{prop}

\subsubsection{The Fefferman-type construction $\Spin(3,4)\hookrightarrow
\Spin(4,4)$}\label{sec-sig44}
Let $\tilde{P}\subset \Spin(4,4)$ be the parabolic subgroup defined as
the stabilizer of the positive ray through a null-vector in
$\mathbb{R}^{4,4}$.  Then, since
as a $\Spin(3,4)$ representation $\rr^{4,4}=\De^{3,4}$, the intersection $P=\tilde{P}\cap \Spin(3,4)$ is precisely the parabolic  introduced in
\ref{sec-sig34}.
Using again that generalized flag manifolds are compact and counting
dimensions one obtains 
\begin{align*} \Spin(3,4)/P \cong \Spin(4,4)/\tilde{P}\cong S^3\times S^3. 
\end{align*} 
Thus, the Fefferman-type construction  associates to a parabolic
geometry of type $(\Spin(3,4),P)$
a parabolic geometry of type $(\Spin(4,4),\tilde{P})$.

\begin{prop} An oriented  generic rank
$3$-distribution on a $6$-manifold naturally induces a conformal spin
structure of signature $(3,3)$ on the manifold.
\end{prop}

\subsubsection{Normality of the induced parabolic geometry}
 For parabolic geometries there is a uniform algebraic normalization condition,  defined in terms of the $P$-equivariant Kostant codifferential 
 \begin{align*}
 \partial^*:\Lambda^2(\mathfrak{g}/\mathfrak{p})^*\otimes\mathfrak{g}\to(\mathfrak{g}/\mathfrak{p})^*\otimes\mathfrak{g}
 \end{align*}
 given on decomposable elements as 
 \begin{align*}
 \partial^*(X\wedge Y\otimes Z)=X\otimes [Y,Z]-Y\otimes[X,Z]-[X,Y]\otimes Z.
 \end{align*}
A parabolic geometry $(\mathcal{G}\to M,\omega)$ of type $(G,P)$ is called \emph{normal} if its curvature function $\kappa:\mathcal{G}\to\Lambda^2(\mathfrak{g}/\mathfrak{p})^*\otimes\mathfrak{g}$ satisfies $\partial^*\circ\kappa=0.$ 
The curvature of a normal parabolic geometry projects to a simpler curvature quantity, the \emph{harmonic curvature} $\kappa_{H}.$ The harmonic curvature takes values in a $G_0$-submodule that is explicitly computable via Kostant's version of the Bott-Borel-Weil theorem \cite{kostant-61}.
 
For applications it will be essential that the conformal parabolic geometries attached to generic distributions are normal.  To verify compatibility of a Fefferman-type construction with normality  is in general a non-trivial problem.
Suppose $\tilde{\omega}$ is the extension 
to $\tilde{\mathcal{G}}=\mathcal{G}\times_{P}\tilde{P}$ of a regular, normal Cartan connection form $\omega$. 
Let $I:\Lambda^2(\mathfrak{g}/\mathfrak{p})^*\otimes\mathfrak{g}\to\Lambda^2(\tilde{\mathfrak{g}}/\tilde{\mathfrak{p}})^*\otimes\tilde{\mathfrak{g}}$
be the induced map from the inclusion $\mathfrak{g}\hookrightarrow\tilde{\mathfrak{g}}$. Then the curvature functions $\kappa:\mathcal{G}\to\Lambda^2(\mathfrak{g}/\mathfrak{p})^*\otimes\mathfrak{g}$
of $\omega$ and $\tilde{\kappa}:\tilde{\mathcal{G}}\to\Lambda^2(\tilde{\mathfrak{g}}/\tilde{\mathfrak{p}})^*\otimes\tilde{\mathfrak{g}}$ of $\tilde{\omega}$ are related by
\begin{align}\label{curvatures}
\tilde{\kappa}(u)=I\circ\kappa(u)\end{align} for all $u\in\mathcal{G}$, and this determines $\tilde{\kappa}$ by equivariance, see \cite{cap-zadnik-chains}. 
We now ask is whether the parabolic geometry 
$(\tilde{\mathcal{G}},\tilde{\omega})$ is normal, i.e., whether $\tilde{\partial}^*\circ\tilde{\kappa}=0$.
By \eqref{curvatures} we can rephrase this as to whether $\kappa$ takes values in the $(G\cap\tilde{P})$-submodule
 $I^{-1}(\mathrm{ker}(\tilde{\partial}^*))\subset\Lambda^2(\tilde{\mathfrak{g}}/\tilde{\mathfrak{p}})^*\otimes\tilde{\mathfrak{g}}$.

For Fefferman-type constructions over the same manifold, i.e., in those cases where $P=(G\cap\tilde{P})$, this problem can be considerably simplified
if one uses the following strong result:
\begin{prop}[\cite{cap-correspondence}]\label{E}
Suppose $\mathbb{E}\subset\mathrm{ker}(\partial^*)\subset\Lambda^2(\mathfrak{g}/\mathfrak{p})^*\otimes\mathfrak{g}$ is a $P$-submodule and consider 
the $G_0$-module $\mathbb{E}_0:=\mathbb{E}\cap\mathrm{ker}(\square). $
Let $(\mathcal{G}\to M,\omega)$ be a  regular, normal parabolic geometry, which is furthermore torsion-free.
Then, if the harmonic curvature $\kappa_{H}$ takes values in $\mathbb{E}_0$ the curvature function $\kappa$ takes values in $\mathbb{E}$.
 \end{prop}

\subsubsection{Normality for $\Spin(3,4)\hookrightarrow
\Spin(4,4)$}
The harmonic curvature $\kappa_{H}$ of a regular, normal parabolic geometry of type $(\Spin(3,4),P)$  takes values in an irreducible $27$-dimensional 
$G_0=\GL_{+}(3,\mathbb{R})$-subrepresentation of $\Lambda^2(\mathfrak{g}/\mathfrak{p})^*\otimes\mathfrak{g}_0$.
This implies that  the geometry is torsion-free, and we may apply the above result to prove:

\begin{prop}\label{Feff33}
 The   Fefferman-type construction associates to a  regular, normal parabolic geometry of type $(\Spin(3,4),P)$
 a normal parabolic geometry of type $(\Spin(4,4),\tilde{P})$.
\end{prop}
\begin{proof}
Let $\tilde{\partial}^*:\Lambda^2\tilde{\mathfrak{p}}_{+}\otimes\tilde{\g}\to\tilde{\mathfrak{p}}_{+}\otimes\tilde{\g}$ be the Kostant codifferential 
describing the conformal normalization condition.
  Since a regular, normal parabolic geometry of type $(\Spin(3,4),P)$ is torsion-free, we can apply Proposition \ref{E}, which shows that to prove
normality of $\omega$  it suffices to prove that $\kappa_{H}$ takes values in the $P$-module $\mathrm{ker}(\tilde{\partial}^*\circ I)$. 

Now $\tilde{\partial}^*\circ I$ is equivariant and thus it either vanishes on  $G_0$-irreducible components, or it is an isomorphism. 
In particular, it must contain the $27$-dimensional irreducible representation where $\kappa_{H}$ takes its values either in its image 
or in its kernel.
The formula for the Kostant codifferential $\tilde{\partial}^*$ shows that if we restrict $\tilde{\partial}^*\circ I$ 
to $\Lambda^2(\mathfrak{g}/\mathfrak{p})^*\otimes\mathfrak{g}_0$  and identify  $(\tilde{\mathfrak{g}}/\tilde{\mathfrak{p}})^*\cong\tilde{\mathfrak{p}}_+$,  
its image is contained in $\tilde{\mathfrak{p}}_{+}\otimes\tilde{\mathfrak{p}}_{+}.$
But $\tilde{\mathfrak{p}}_{+}$ decomposes as a
$G_0$-representation as $(\mathbb{R}^3)^*\oplus\Lambda^2(\mathbb{R}^3)^*$, and therefore
$\tilde{\mathfrak{p}}_{+}\otimes\tilde{\mathfrak{p}}_{+}$ cannot
contain a $27$-dimensional irreducible summand.
\end{proof}

\subsubsection{Normality for $G_2\hookrightarrow \Spin(3,4)$}
We used similar arguments  in \cite{mrh-sag-rank2} to prove that the extension
of a regular, normal Cartan connection $\omega$ of type $(G_2,P)$ to a  Cartan connection  $\tilde{\omega}$ of type $(\SO(3,4),\tilde{P})$ is again normal.
Note that the arguments in the proof do not depend on the choice of groups $(\Spin(3,4),\tilde{P})$ or $(\SO(3,4),\tilde{P})$, respectively.
\begin{prop}\label{Feff23}
 The   Fefferman-type construction associates to a  regular, normal parabolic geometry of type $(G_2,P)$
 a normal parabolic geometry of type $(\Spin(3,4),\tilde{P})$.
\end{prop}

\subsubsection{The twistor spinors of generic $2$ and $3$ distributions}\label{sec-fefftwistors}
\begin{thm}\label{prop-twiofdi}
 The Fefferman-type constructions  \ref{sec-sig34} and \ref{sec-sig44} for generic $2$ and $3$ distributions determine generic twistor spinors. The kernels of these twistor spinors
recover the $2$ or $3$ distribution.
\end{thm}
\begin{proof}
Let $\mathcal{D}$ be a generic rank $2$ or $3$ distribution,
$(\mathcal{G},\omega)$ the associated parabolic geometry of type $(G,P)$ and
$(\tilde{\G},\tilde{\omega})$ the conformal spin geometry obtained via the Fefferman-type construction. Then the spin tractor bundle of $\mc{C}$\ is, for $(p,q)=(2,3)$\ resp. $(3,4)$,
\begin{align}
 \mc{S}=\ti\G\times_{\ti P}\De^{p+1,q+1}=\G\times_{P}\De^{p+1,q+1}.
\end{align}
Since $G\subset\Spin(p+1,q+1)$ is the isotropy subgroup of a non-null element, say $X\in\De^{p+1,q+1}$, the constant function $\G\to\De^{p+1,q+1}$
onto this element defines a  spin-tractor  $\mb{X}\in\Gamma(\mc{S}).$ Propositions \ref{Feff23} and \ref{Feff33} imply that the
spin tractor connection $\nabla^{\mc{S}}$ is induced from the canonical Cartan connection $\omega\in\Omega^1(\mathcal{G},\mathfrak{g})$ 
for the distribution $\mathcal{D}.$ Thus, $\nabla^{\mc{S}}\mb{X}=0$.
By Proposition \ref{prop-spin}, $\mb{X}$ corresponds via the map \eqref{twisplit} to 
a twistor spinor $\chi$. 
Since $X$ is non-null for the bilinear form on $\De^{p+1,q+1}$, the corresponding spinor tractor is non-null for the induced form on the tractor bundle. 
By \eqref{formb34}\ resp. \eqref{formb44} this is equivalent to 
$\mb{b}_{3,4}(\chi,\dirac\chi)\not=0$\ resp. $\mb{b}_{4,4}(\chi,\dirac\chi)\not=0$
 for the underlying twistor spinor, and thus $\chi$\ is generic.

Moreover $\chi$ has kernel $\mathcal{D}$. This follows from the fact that via the identification $TM\cong \mathcal{G}\times_{P}\g/\p$ determined
by the Cartan connection $\omega\in\Omega^1(\mathcal{G},\mathfrak{g}),$ the distribution corresponds to the subbundle $\mathcal{G}\times_{P}\g^{-1}/\p$,
and from the descriptions of the gradings from Propositions \ref{propg2} and \ref{propso34}.
\end{proof}

Note that it is  not obvious that the map  that assigns to a generic twistor spinor its kernel
and the map \ref{sec-fefftwistors} from distributions to twistor spinors coming from Fefferman-type constructions  discussed above are inverse bijections:
 a priori we don't know  that
\emph{all} generic twistor spinors in the right signatures are induced from generic distributions. In order to obtain a characterization of the
conformal spin structures associated to 
generic rank $2$ and $3$ distributions by generic twistor spinors we will invoke the holonomy characterizations to be discussed below.

\subsection{Conformal holonomy characterization of the Fefferman-type spaces}

\subsubsection{Holonomy of the conformal structures associated to generic $2$ and $3$ distributions}
We have seen in Proposition \ref{prop-twiofdi} that the conformal spin structures induced by generic distributions carry parallel
spin-tractors. Thus, they have reduced conformal holonomy:
Conformal spin structures
associated to oriented generic rank $2$-distributions in dimension $5$
have conformal holonomy contained in $G_2\subset \Spin(3,4)$, and
conformal spin structures associated to oriented generic rank
$3$-distributions in dimension $6$ have conformal holonomy contained
in $\Spin(3,4)\subset \Spin(4,4).$

Let $(\ti\G,\ti\om)$\ be a regular, normal parabolic geometry of type $(\Spin(p,q),P)$\
 encoding a conformal structure $\mc{C}$\ on $M$. 
 Then $\mathrm{Hol}(\mc{C})$\ was defined in Definition \ref{confhol}
as the holonomy of the spin tractor connection $\na^{\mc{S}}$.
Since $\na^{\mc{S}}$\ is the induced connection
from the extended normal Cartan connection $\ti\om'\in\Om^1(\ti\G,\so(p+1,q+1))$ (cf. subsection \ref{sec-tractor}), we have that
$\Hol(\mc{C})=\Hol(\na^{\mc{S}})=\Hol(\ti\om')\subset\Spin(p+1,q+1)$.
It will be useful for our purposeses of reversing the Fefferman-type constructions from above to see the conformal holonomy reductions from this
viewpoint:

Suppose $(\mathcal{G},\omega)$ is a regular normal parabolic geometry
of type $(G,P)$, and suppose there is a Fefferman-type construction  that gives rise to a normal
parabolic geometry $(\tilde{\mathcal{G}},\tilde{\omega})$ of type
$(\Spin(p+1,q+1),\tilde{P})$ over the same manifold.  Let
$\mathcal{G}'=\mathcal{G}\times_{P}G$ be the extended $G$-principal bundle of $\G$\
and $\omega'$ the principal connection obtained by extension.  Then we
have the commuting diagram of inclusions
\begin{align*} \xymatrix{ (\G',\om') \ar@{^{(}->}[r] &
(\ti\G',\ti\om') \\ (\G,\om) \ar@{^{(}->}[u] \ar@{^{(}->}[r] & (\ti
\G,\ti\om) \ar@{^{(}->}[u] }  
\end{align*} which shows that
$\tilde{\omega}'\in\Omega^1(\tilde{\mathcal{G}}',\mathfrak{so}(p+1,q+1))$
reduces to the $G$-principal bundle connection $\omega'\in\Omega^1(\mathcal{G}',\mathfrak{g})$ and thus
$\mathrm{Hol}(\mc{C})=\mathrm{Hol}(\tilde{\omega}')=\mathrm{Hol}(\omega')\subset G.$

\subsubsection{Holonomy characterizations}
Conversely, let $(\tilde{\mathcal{G}},\tilde{\omega})$ be a parabolic
geometry of type $(\Spin(p+1,q+1),\tilde{P})$ such that the conformal
holonomy group $\mathrm{Hol}(\tilde{\omega}')$ is contained in $G$ and
suppose we have a parabolic subgroup $P\subset G$ with $G/P\cong
\Spin(p+1,q+1)/\tilde{P}.$ Then, since we have a holonomy reduction,
$\tilde{\mathcal{G}}'$ reduces to a $G$-principal bundle
$\mathcal{G}'$ and $\tilde{\omega}'$ reduces to
$\omega'\in\Omega^1(\mathcal{G}',\mathfrak{g})$. Using that
$G/P\cong \Spin(p+1,q+1)/\tilde{P},$ one can show that $\mathcal{G}'$
intersects with $\tilde{\mathcal{G}}$ in a $P$-principal bundle
$\mathcal{G}$ and $\omega'$ restricts to a Cartan connection
$\omega\in\Omega^1(\mathcal{G},\mathfrak{g})$. See  \cite{mrh-thesis} and \cite{mrh-sag-rank2}  for details.

A normal conformal Cartan connection $\tilde{\omega}$ is torsion-free,
i.e. the curvature function $\tilde{\kappa}:\tilde{\mathcal{G}}\to\Lambda^2(\tilde{\mathfrak{g}}/\tilde{\mathfrak{p}})^*\otimes\tilde{\mathfrak{g}}$  takes values in th $P$-submodule $\Lambda^2(\tilde{\mathfrak{g}}/\tilde{\mathfrak{p}})^*\otimes\tilde{\mathfrak{p}}$. Since
$\mathfrak{p}=\tilde{\mathfrak{p}}\cap \mathfrak{g}$ and because of
the relation of the curvatures of $\omega$ and $\tilde{\omega}$, 
as in \eqref{curvatures}, this implies that $\omega$ is torsion-free
and thus regular. This means that the geometry $(\mathcal{G},\omega)$
obtained by reduction as explained in the previous paragraph induces
an underlying generic distribution $\mathcal{D}$.  

It remains to prove that  this distribution $\mathcal{D}$ induces via the Fefferman-type construction the conformal structure encoded in 
$(\tilde{\mathcal{G}},\tilde{\omega})$. For this we need to see that $\omega$ and the normal Cartan connection $\omega_{N}$ for $\mathcal{D}$, which is a priori different, induce the same conformal structure.
Now, the conformal structure induced by a geometry $(\mathcal{G},\omega)$ does not depend on the entire Cartan connection, but only on 
the   isomorphism $TM\cong \mathcal{G}\times_{P}\mathfrak{g}/\mathfrak{p}$ defined by the Cartan connection.
If the difference $\omega-\omega_{N}$, seen as a function $\mathcal{G}\to\Lambda^2(\mathfrak{g}/\mathfrak{p})^*\otimes\mathfrak{g}$
 takes values in $\Lambda^2(\mathfrak{g}/\mathfrak{p})^*\otimes\mathfrak{p}$, then the two induce the same conformal structure. 
It is verified in  Proposition 4.1 of \cite{armstrong-3distributions} that this is the case for generic rank $3$-distributions in 
dimension $6$:  Since $\omega$ is torsion-free, $\kappa$ takes values in maps $\Lambda^2(\mathfrak{g}/\mathfrak{p})^*\otimes\mathfrak{p}$ 
of homogeneity $\geq 2$, where the homogeneity is defined with respect to the grading of $\mathfrak{g}$. Since $\partial^*$ preserves 
homogeneities, the same is true for $\partial^*\kappa,$ and then also the difference $\omega-\omega_{N}$ maps into 
homogeneity $\geq 2$, see Proposition 3.1.13 in \cite{cap-slovak-par}. If $\mathfrak{g}$ is $2$-graded, as it is the case for generic rank $3$ distributions, this already implies that
 $\omega-\omega_{N}$ takes values in $\Lambda^2(\mathfrak{g}/\mathfrak{p})^*\otimes\mathfrak{p}$.  For generic rank $2$ distributions a similar argument applies if one first directly 
verifies that a certain curvature component of $\kappa$ vanishes, see \cite{katja-thesis} and \cite{mrh-sag-rank2}.

Summarizing we obtain:
\begin{thm}\label{thm-holred1}
The Fefferman-type construction for oriented generic rank $2$-distributions produces exactly those conformal spin structures of signature $(2,3)$ whose conformal holonomy is contained in
$G_2$.
\end{thm}
\begin{thm}\label{thm-holred2}
The Fefferman-type construction for oriented generic rank $3$-distributions produces exactly those conformal spin structures of signature $(3,3)$ whose conformal holonomy is contained in $\Spin(3,4)$.
\end{thm}

We remark that if $\Hol(\mc{C})\subset G$\ is a proper subgroup, there
may be several holonomy reductions yielding different generic distributions.

\begin{cor}
  A conformal spin structure $(M,\mc{C})$\ of signature $(2,3)$ or $(3,3)$\ is induced by a generic
  $2$- resp $3$-distribution $\mc{D}\subset TM$\  via a Fefferman-type construction if and only if $(M,\mc{C})$\ carries
  a generic twistor spinor $\chi$.
\begin{proof}
The result is an immediate corollary of  Theorems \ref{thm-holred1} and \ref{thm-holred2}: The general holonomy correspondence between holonomy invariant elements and parallel sections implies that holonomy reductions to
$G_2$ respectively $\Spin(3,4)$ correspond to the existence of parallel non-null spin tractors. Parallel spin-tractors
are equivalent to twistor spinors  via \eqref{twisplit}, and non-isotropy of the
spin tractor translates into the genericity conditions expressed via
the canonical forms $\mb{b}_{2,3}$\ resp. $\mb{b}_{3,3}$.

\end{proof}

\end{cor}

\begin{rema}
  \begin{enumerate}
  \item In particular, we have shown that the conformal structure $\mc{C}$\
        and its generic twistor spinor $\chi$\ are completely determined
        by the generic distribution $\mc{D}_{\chi}=\ker\ga\chi$.
  \item The existence of this twistor spinor should have interesting
consequences for the ambient metric construction of
the associated conformal structures. Indeed, for
certain examples of generic $2$-distributions Leistner and
Nurowski \cite{leistner-nurowski-g2ambient} have found
a corresponding parallel spinor on the ambient metric of the $(2,3)$-conformal
structure. 
  \end{enumerate}
\end{rema}

\section{Decompositions of conformal Killing fields}\label{ckfdecomp}
The main result of this section is an explicit decomposition of infinitesimal conformal automorphisms via a transversal twistor spinor. 
Furthermore, we provide formulae relating almost Einsetin structures with a subset of twistor spinors.

\subsection{Infinitesimal  automorphisms of conformal spin structures associated to $2$ and $3$ distributions}\label{secinfinitaut}
Infinitesimal automorphisms of a conformal structure $\mc{C}$ are
\emph{conformal Killing fields}, i.e., vector fields $\xi\in\X(M)$
such that $\mathcal{L}_{\xi}g=f g$ for some $g\in\mc{C}$ and some $f\in
\cinf(M)$. Infinitesimal automorphisms of a distribution $\mathcal{D}$
are vector fields whose Lie derivatives preserve the distribution,
i.e., $\xi\in\X(M)$ such that
$\mathcal{L}_{\xi}\eta=[\xi,\eta]\in\Gamma(\mathcal{D})$ for all
$\eta\in\Gamma(\mathcal{D})$.

For the structures we are interested in  infinitesimal
automorphisms correspond to infinitesimal automorphisms of the
associated Cartan geometries. Our decomposition
result is based on a description of infinitesimal automorphisms of
a Cartan geometry $(\mathcal{G},\omega)$ as sections of the \emph{adjoint tractor bundle}
$\mathcal{A}M=\mathcal{G}\times_P\g$ that are parallel with respect to
the prolongation connection
\begin{equation}\label{infautc} \hat{\nabla}_{\xi}s=\nabla_{\xi}
s-\kappa(\xi,\Pi(s)),
\end{equation} see \cite{cap-infinitaut}.  Here
$\Pi:\mathcal{A}M=\mathcal{G}\times_{P}\mathfrak{g}\to\mathcal{G}\times_{P}\mathfrak{g}/\mathfrak{p}=TM$
is the natural projection and $\kappa$ is the curvature of $\omega$
viewed as an element of $\Omega^1(M,\mathcal{A}M)$.

Let $\mathfrak{g}\subset\tilde{\mathfrak{g}}$ be either the inclusion
$\mathfrak{g}_2\subset\mathfrak{so}(3,4)$, or
$\mathfrak{so}(3,4)\subset\mathfrak{so}(4,4)$. Then, as a
representation of $G$, i.e. $G_2$ or $\SO(3,4)$, we have a
decomposition
\begin{equation*}
\tilde{\mathfrak{g}}=\mathfrak{g}\oplus\mathbb{R}^{3,4}.
\end{equation*} It follows that the conformal adjoint tractor bundle
\begin{equation*}
\tilde{\mathcal{A}}M=\tilde{\mathcal{G}}\times_{\tilde{P}}\tilde{\mathfrak{g}}=\mathcal{G}\times_{P}\tilde{\mathfrak{g}}
\end{equation*} of a conformal structure associated to a generic
distribution decomposes into $\A M$\ and a $7$-dimensional
complementary bundle
$\mathcal{V}=\mathcal{G}\times_{P}\mathbb{R}^{3,4}$. The tractor
connection is natural and decomposes accordingly.

\begin{prop}\label{prop-adecomp}Suppose
$s=s_1+s_2\in\Gamma(\tilde{\mathcal{A}}M)=\Gamma(\mathcal{A}M)\oplus\Gamma(\mathcal{V})$
is the decomposition of a conformal adjoint tractor according to the
decomposition of $\tilde{\mathcal{A}}M$. Then $s$ is parallel with
respect to the infinitesimal automorphism connection on
$\tilde{\mathcal{A}}M$ if and only if $s_1$ is parallel with respect
to the infinitesimal automorphism connection on $\mathcal{A}M$ and
$s_2$ is parallel with respect to the tractor connection on
$\mathcal{V}$.
\end{prop} 
The proof  is discussed in detail the $G_2$-case in
\cite{mrh-sag-rank2}, the $\Spin(3,4)$-case is completely analogous.
Briefly, one uses that the curvature of the conformal structures associated
to such distributions takes indeed values in $\mathcal{A}M$, see
\cite{mrh-sag-rank2}, and that adjoint tractors in $\tilde{A}M$
are parallel for the tractor connection insert trivially into the
curvature, see \cite{gover-lapl_einstein},\cite{mrh-sag-rank2}.

\subsection{Explicit splitting formulas}
For the explicit decompositions in sections \ref{decomp23} and \ref{decomp33}
it will be necessary to collect some explicit differential
splitting formulas, which are particular instances 
of \emph{BGG-splitting operators} \cite{BGG-2001, BGG-Calderbank-Diemer}.

\subsubsection{Splittings of conformal Killing fields}
 We have discussed in section \ref{secinfinitaut} that
an infinitesimal symmetry of a conformal structure $(M,\mc{C})$\ is equivalent to
an adjoint tractor $s\in\Ga(\ti\A M)$\ that is parallel with respect
to \eqref{infautc}. To make the relation between the conformal Killing
fields $\xi\in\X(M)$\ and the adjoint tractor $s\in\Ga(\ti\A M)$\ explicit,
one employs the canonical differential splitting operator
$L_0^{\ti\A}$, that is given by \cite{mrh-bgg,mrh-thesis}
\begin{align}\label{L0La2}
  &L_0^{\ti A}:\X(M)\goesto\Ga(\La^2\mc{T}),\\
  &\xi^a\mapsto
  \begin{pmatrix}
        \begin{pmatrix}
          -\frac{1}{2n}D^pD_p\xi_{a}
          +\frac{1}{2n}D^pD_{a}\xi_{p} 
          +\frac{1}{n^2}D_{a}D^p\xi_{p}
          \\
          +\frac{2}{n}\mr{P}^p_{\; a}\xi_{p}
          -\frac{1}{n}J\xi_{a}
    \end{pmatrix}
    \\
    {D}_{[a_0}\xi_{a_1]} \; | \;
    -\frac{1}{n}\bg^{pq}{D}_p\xi_{q}
    \\
    \xi_{a}
  \end{pmatrix}.
\end{align}
  It follows from (\cite{cap-infinitaut}) that
\begin{lem}
  $\xi\in \X(M)$\ is a conformal Killing field if and only if $L_0^{\ti\A}(\xi)$\
  is parallel with respect to the connection \eqref{infautc}.
\end{lem}

We already have interpretations of parallel sections of the spin tractor
bundle as twistor spinors  and of parallel sections (with respect to a modified
connection) of the adjoint tractor bundle as conformal Killing fields.
We now recall the interpretation of the parallel sections of
the conformal standard tractor bundle:

\subsubsection{The splittings of almost Einstein scales}\label{BGGsplittings}

Let $s\in\Ga(\mc{T})$, then with respect to a $g\in\mc{C}$\
one has (recall section \ref{secstd})
$
  [s]_g=
  \begin{pmatrix}
    \rh \\
    \ph_a \\
    \si
  \end{pmatrix}
  \in \Ga([\mc{T}]_g)
$,
and the explicit transformation behaviour of $[s]_g$\ under a change
of metric shows that one has a canonical projection
$\Pi^{\mc{T}}_0:\mc{T}\goesto\ce[1]$.
Moreover, it can be seen from the definition of $\na^{\mc{T}}$\ and
some simple differential consequences, \cite{thomass}, that
the section $s\in\mc{T}$\ is $\na^{\mc{T}}$-parallel if and only if
$\si=\Pi_0^{\mc{T}}$\ satisfies 
the equation
\begin{align}\label{equaes}
  (D_aD_b\si+\mr{P}_{ab}\si)_0=0.
\end{align}
Given a solution $\si$\ of \eqref{equaes}, the
corresponding $\na^{\mc{T}}$-parallel tractor
is obtained by the splitting operator \cite{thomass}
\begin{align}\label{splitStd} L_0^{\mc{T}}:\ce[1]&\goesto
\Ga(\mc{T}),\\\notag \si&\mapsto
  \begin{pmatrix} -\frac{1}{n}g^{pq}(D_{pq}\si+\PP_{pq}\si) \\ D\si \\
\si
  \end{pmatrix}.
\end{align}

A solution $\si\in\Ga(\ce[1])$\ of \eqref{equaes}
has been termed an \emph{almost Einstein scale} by
R. Gover, \cite{gover-aEs},
since \eqref{equaes}\ turns out to be equivalent
to $\si^{-2}\bg$\ being \emph{Einstein}\ wherever $\si$\ is
non-zero.

\subsection{Explicit decomposition in signature $(2,3)$}\label{decomp23}
We start with conformal spin structures
of signature $(2,3).$ Given a generic twistor spinor $\chi$
and its tractor spinor $\mb{X}=L_0^{\mc{S}}(\chi)$, we consider the orthogonal
complement $\mb{X}^{\perp}$ of $\rr\mb{X}$\ in $\Ga(\mc{S})$\ with respect to $\mb{B}_{3,4}$. It
is easy to see that a twistor spinor $\eta\in\Ga(S[\frac{1}{2}])$\
splits into this complementary space if and only if
$\mb{b}_{2,3}(\eta,\dirac\chi)+\mb{b}_{2,3}(\chi,\dirac\eta)=0$. The space of 
 twistor spinors satisfying this equation shall be denoted by
$\mb{Tw}_{\mc{C}}^{\bot}(\chi)\subset\mb{Tw}_{\mc{C}}\subset\Ga(S[\frac{1}{2}])$.

\begin{lem}\label{lemaesckf} For a fixed generic twistor spinor $\chi$, we have a bijective correspondence
between almost Einstein scales and  twistor spinors $\eta$ such that $\mb{b}_{2,3}(\eta,\dirac\chi)+\mb{b}_{2,3}(\chi,\dirac\eta)=0$.
 An almost Einstein scale $\si\in\ce[1]$\ is mapped to
  \begin{align*}
\frac{2}{5}\si\dirac\chi+(D\si)\cdot\chi\in\mb{Tw}_{\mc{C}}^{\bot}(\chi).
  \end{align*}
and a twistor spinor $\eta\in\mb{Tw}_{\mc{C}}^{\bot}(\chi)$ is mapped to the almost Einstein scale $\sigma=\mb{b}_{2,3}(\chi,\eta)$.
\end{lem}

\begin{proof} Suppose $X\in\Delta^{3,4}$ with $B_{3,4}(X,X)\neq 0$ is 
stabilized by $G_2$. Consider the map
$\mathbb{R}^{3,4}\to\Delta^{3,4}$ given by Clifford multiplication on
$X$. Then this is a non-zero, $G_2$-equivariant map. Moreover,
$\mathbb{R}^{3,4}$ is an irreducible $G_2$-representation, and thus
the map must be an isomorphism onto its image, which is the orthogonal
complement $X^{\perp}$ to $\mathbb{R}X$ in $\Delta^{3,4}$ with respect
to the invariant bilinear form $B_{3,4}$. The inverse maps a spinor $Y$ to the unique element 
$v\in\mathbb{R}^{3,4}$ such that $h_{3,4}(v,w)B_{3,4}(X,X)=B_{3,4}(w\cdot X, Y)$ for all $w\in\mathbb{R}^{(3,4)}.$

Passing to associated bundles yields an identification of the standard tractor bundle $\mathcal{T}$ with
$\mathbf{X}^{\perp}$. In order to obtain the explicit formulas relating almost Einstein scales with the subset $\mb{Tw}_{\mc{C}}^{\bot}(\chi)$ of the space of twistor
spinors, we use the formulas $\eqref{splitStd}$, $\eqref{twisplit}$ and $\eqref{traCli}$ for the splitting operators $L_0^{\mc{T}}:\ce[1]\goesto
\Ga(\mc{T})$ and  $L_0^{\mc{S}}:\Ga(S[\frac{1}{2}])\goesto\Ga(\mc{S}) $ and for the Clifford  multiplication $\ga:\mc{T}\t\mc{S}\goesto\mc{S}$.

\end{proof}

We proceed to the decomposition of infinitesimal automorphisms in
terms of $\chi$.

\begin{prop}\label{thm-ckf23} Given a conformal spin structure of
signature $(2,3)$ and a generic twistor spinor
$\chi\in\Ga(S[\frac{1}{2}])$, the
space of conformal Killing fields decomposes into the space of almost
Einstein scales and the space of infinitesimal automorphisms of the
corresponding rank $2$-distribution.  Explicitly, for some $g\in \mc{C}$, the almost Einstein
scale part of a conformal Killing field $\xi\in\X(M)$\ is given by
  \begin{align*}
\si=\mb{b}_{2,3}(\chi,-\frac{4}{5}\xi\cdot\dirac\chi+(D_{[a}\xi_{b]})\cdot\chi)\in\ce[1].
  \end{align*} Conversely, an almost Einstein scale $\si\in\ce[1]$\ is
mapped to a conformal Killing field
  \begin{align*}
\xi_a=\mb{b}_{2,3}(\ga_a\chi,\frac{2}{5}\si\dirac\chi+(D\si)\cdot\chi)\in\ce_a[2]=\X(M)
  \end{align*}
\end{prop}

\begin{proof} Proposition \ref{prop-adecomp} implies that conformal
Killing fields decompose into infinitesimal automorphisms of the
distributions and almost Einstein scales. Algebraically, the
projection $\so(3,4)=\Lambda^2\mathbb{R}^{3,4}\to\mathbb{R}^{3,4}$ is
given by  action of $\Lambda^2\mathbb{R}^{3,4}$ on the non-null spinor $X$ composed with the
isomorphism $X^{\perp}\cong\mathbb{R}^{3,4}$ from Lemma
\ref{lemaesckf}. 
The inverse of the map $\Lambda^2\mathbb{R}^{3,4}\to \mathbb{R}^{3,4}$ assigns to $w\in\mathbb{R}^{3,4}$ the unique element 
$\phi\in\Lambda^2\mathbb{R}^{3,4}=\Lambda^2(\mathbb{R}^{3,4})^*$ such that $\phi(u,v) B_{3,4}(X,X)=B_{3,4}(u\cdot v\cdot X,w\cdot X)$ for all $u,v\in\mathbb{R}^{3,4}$. 

The explicit decomposition in terms of $\chi$\ is then obtained using the algebraic maps and 
 the differential  BGG-splitting operators.
\end{proof}

\subsection{Explicit decomposition in signature $(3,3)$}\label{decomp33}
Suppose we have a conformal spin
structure of signature $(3,3)$, a generic twistor spinor
$\chi\in\Ga(S_{+}[\frac{1}{2}])$
and its tractor spinor $\mb{X}=L_0^{\mc{S}}(\chi)$. Again, we use the
notation $\mb{Tw}_{\mc{C}}^{\bot}(\chi)\subset\mb{Tw}_{\mc{C}}\subset\Ga(S_{+}[\frac{1}{2}])$
for the space of  twistor spinors satisfying
$\mb{b}_{3,3}(\eta,\dirac\chi)+\mb{b}_{3,3}(\chi,\dirac\eta)=0$; these are mapped
via the splitting operator \eqref{twisplit} into the orthogonal complement
$\mb{X}^{\perp}$.

\begin{prop}\label{thm-ckf33} Given a conformal spin structure of
signature $(3,3)$ and a generic twistor spinor
$\chi\in\Ga(S[\frac{1}{2}])$, 
the space of conformal Killing fields decomposes into the space of
$\mb{Tw}_{\mc{C}}^{\bot}(\chi)$ and the space of
infinitesimal automorphisms of the corresponding rank
$3$-distribution.  Explicitly, for a $g\in \mc{C}$, a conformal Killing field
$\xi\in\X(M)$\ is mapped to the  twistor spinor
  \begin{align*}
\eta=-\frac{2}{3}\xi\cdot\dirac\chi+(D_{[a}\xi_{b]})\cdot\chi+\frac{1}{6}(\de\xi)\chi\in\mb{Tw}_{\mc{C}}^{\bot}(\chi)
  \end{align*} and a twistor spinor
$\eta\in\mb{Tw}_{\mc{C}}^{\bot}(\chi)$\ is mapped to the conformal Killing field
  \begin{align*} \xi_a=\mb{b}_{3,3}(\chi,\ga_a\eta)\in\ce_a[2]=\X(M).
  \end{align*}
\end{prop}

\begin{proof} The map
$\mathfrak{so}(4,4)=\Lambda^2\mathbb{R}^{4,4}\to\Delta_{+}^{4,4}$
given by action on the $\Spin(3,4)$-invariant element
$X\in\Delta_{+}^{4,4}$ provides an identification of the
subrepresentation $\mathbb{R}^{3,4}\subset\mathfrak{so}(4,4)$ with
$X^{\perp}\subset\Delta_{+}^{4,4}.$ Employing Proposition
\ref{prop-adecomp} we thus have a decomposition of conformal Killing
fields into infinitesimal automorphisms of the distribution and
 twistor spinors $\mb{Tw}_{\mc{C}}^{\bot}(\chi)$. 
Again, the explicit maps are obtained 
via  
simple computations using differential splitting
formulas \eqref{L0La2} and \eqref{twisplit}.
 \end{proof}

In this case almost Einstein scales don't correspond to infinitesimal
automorphisms, but again, they can be identified with a subset of
twistor spinors. For this, note that Clifford multiplication on the non-null spinor
$X\in\Delta_{+}^{4,4}$ defines an isomorphism $\mathbb{R}^{4,4}\cong
\Delta^{4,4}_{-}$ and therefore have:

\begin{prop}\label{prop-aes} There is a bijective correspondence
between almost Einstein scales and negative twistor spinors.  An
almost Einstein scale $\si\in\ce[1]$\ is mapped to the negative
twistor spinor
  \begin{align*}
\eta=\frac{1}{3}\si\dirac\chi+(D\si)\cdot\chi\in\Ga(S_-[\frac{1}{2}]).
  \end{align*} Conversely, a negative twistor spinor
$\eta\in\Ga(S_-[\frac{1}{2}])$\ is mapped to the almost Einstein scale
$\si=\mb{b}_{3,3}(\chi,\eta)\in\ce[1]$.
\end{prop}

\subsection{Remark on the normal conformal Killing forms  induced by the generic twistor spinors}
Fixing the a generic twistor spinor $\chi$ gives rise to more  maps from twistor spinors to normal conformal Killing $k$-forms: 
 In signature $(2,3)$ we have an inclusion from twistor spinors into normal conformal Killing $2$-forms. In particular, a generic twistor spinor $\chi$ gives rise to\linebreak $\phi_2:=\mb{b}_{2,3}(\chi,\ga_{[a}\ga_{b]}\chi)\in\ce_{[ab]}[3]$.
In \cite{mrh-sag-rank2} we use the existence of such a conformal Killing $2$-form to characterize the conformal structures associated to $2$-distributions.

In signature $(3,3)$ we have an inclusion from negative twistor spinors into normal conformal Killing $2$-forms and an inclusion of positive twistor spinors into normal conformal Killing $3$-forms. In particular, in this case we always have a non-trivial normal conformal Killing $3$-form, which is given by\linebreak
$\phi_3:=\mb{b}_{3,3}(\chi,\ga_{[a}\ga_{b}\ga_{c]}\chi)\in\ce_{[abc]}[4]$.

  Since $\phi_2$\ and $\phi_3$\ correspond to to pure spinors they are
decomposable and therefore insertion into the form has a $3$-dimensional
kernel in both cases. In signature $(3,3)$\ this is already the
canonical $3$-distribution associated to $\phi_3$, and in signature $(2,3)$\
the canonical $2$-distribution is formed by the intersection of
the kernel of $\bg$\ with the kernel of $\phi_2$.

\end{document}